\documentclass[a4paper,reqno]{amsart}

\usepackage[utf8]{inputenc}
\usepackage{amssymb}
\usepackage{enumitem}
\usepackage{color}
\usepackage{mathrsfs}
\usepackage{mathtools}
\setlist[enumerate]{label=\emph{(\roman*)}}

\usepackage{hyperref}

\newtheorem{theorem}{Theorem}[section]
\newtheorem{corollary}[theorem]{Corollary}
\newtheorem{lemma}[theorem]{Lemma}
\newtheorem{proposition}[theorem]{Proposition}

\theoremstyle{definition}
\newtheorem{definition}[theorem]{Definition}
\newtheorem{remark}[theorem]{Remark}

\numberwithin{equation}{section}

\newcommand\R{\mathbb{R}}
\newcommand\e{\epsilon}

\newcommand\T{\mathbb{T}}
\newcommand{\la}{\langle}
\newcommand{\ra}{\rangle}

\begin{document}

\parindent=0pt

\title[three-dimensional Euler equation with helical symmetry]
{Global well-posedness of weak solutions to the incompressible Euler equations with helical symmetry in $\R^3$}

\author[D. Guo]{Dengjun Guo}
\address{School of Mathematical Sciences,
University of Science and Technology of China, Hefei 230026, Anhui, China}
\email{guodeng@mail.ustc.edu.cn}
\author[L. Zhao]{Lifeng Zhao}
\address{School of Mathematical Sciences,
University of Science and Technology of China, Hefei 230026, Anhui, China}
\email{zhaolf@ustc.edu.cn}
\email{}

\thanks{L. Zhao is supported by NSFC Grant of China No. 12271497, No. 12341102 and the National Key Research and Development Program of China No. 2020YFA0713100.
}
\thanks{\textbf{Data Availability Statements:} Data sharing not applicable to this article as no datasets were generated or analysed during the current study.}
\begin{abstract}
We consider the three-dimensional incompressible Euler equation
\begin{equation*}\left\{\begin{aligned}
&\partial_t \Omega+U \cdot \nabla \Omega-\Omega\cdot \nabla U=0  \\
&\Omega(x,0)=\Omega_0(x)
\end{aligned}\right.
\end{equation*}
in the whole space $\R^3$. Under the assumption that $\Omega^z$ is helical and in the absence of vorticity stretching, we prove the global well-posedness of weak solutions in $L^1_1 \bigcap L^{\infty}_1(\R^3)$. The vortex transport formula is also obtained in our article.
\end{abstract}

\maketitle

\section{Introduction}

\quad We consider the three-dimensional incompressible Euler equation in $\R^3$,
\begin{equation}\label{eq 3euler velocity}\left\{ \begin{aligned}
&\partial_t U + U \cdot \nabla U+ \nabla P=0 \\
&\nabla \cdot U=0 \\
&U(x,0)=U_0(x).
\end{aligned}\right.
\end{equation}

    \quad The equation describes the motion of an ideal incompressible fluid in $\R^3$ with initial velocity $U_0(x)$. Here $U=(U^1,U^2,U^3):\R^3\times \R \to \R^3$ represents the velocity and $P:\R^3\times \R \to \R$ represents the scalar pressure which can be determined by the velocity $U$ via the incompressibility condition.

  \quad The motion of the fluid can also be described by its vorticity $\Omega(x,t):=\nabla \wedge U(x,t)$, where
  $$
  (a_1,a_2,a_3)\wedge (b_1,b_2,b_3):=(a_2b_3-a_3b_2,a_3b_1-a_1b_3,a_1b_2-a_2b_1).
  $$

   Moreover, the vorticity satisfies
\begin{equation}\label{eq 3euler}\left\{\begin{aligned}
&\partial_t \Omega+U \cdot \nabla \Omega-\Omega\cdot \nabla U=0 \\
&\Omega(x,0)=\nabla \wedge U_0(x):=\Omega_0(x),
\end{aligned}\right.
\end{equation}
which is called the vorticity-stream formulation of the Euler equations.
The velocity $U$ can be recovered from $\Omega$ by the well-known Biot-Savart law
\begin{equation}\label{eq biot savart law}U(x)=-\frac{1}{4\pi}\int_{\R^3}\frac{x-y}{|x-y|^3}\wedge \Omega(y)\,dy.\end{equation}
\quad The local well-posedness of the three-dimensional Euler equation has been studied in \cite{MaB} for initial data $U_0 \in H^m$ and $m \ge 4$. However, the global existence of smooth solutions remains open due to the strong nonlinearity of the vorticity stretching term $\Omega \cdot \nabla U$, see  \cite{BT} or \cite{MaB} for more references.

  \quad Unlike the three-dimensional case, the two-dimensional incompressible Euler equation
  \begin{equation*}\left\{ \begin{aligned}
&\partial_t w + u \cdot \nabla w=0 \\
&u=\nabla^{\perp} (-\Delta)^{-1} w \\
&w(x,0)=w_0(x)
\end{aligned}\right.
\end{equation*}
   is globally well-posed in $L^1 \bigcap L^{\infty}$ due to the absence of the vorticity stretching term, see \cite{MaB} and \cite{Yud}. Moreover, the two-dimensional incompressible Euler equation can sometimes be viewed as a transport equation and the vortex transport formula holds for solutions with bounded $L^1 \bigcap L^{\infty}$ norm. That is, let $X(\alpha,t)$ solves the ODE \begin{equation*}\begin{cases}\begin{aligned}
&\frac{dX(\alpha,t)}{dt}= u(X(\alpha,t),t)\\
&X(\alpha,0)=\alpha
\end{aligned}\end{cases}\end{equation*} for all $\alpha \in \R^3$, then $w(X(\alpha,t),t)=w_0(\alpha)$. This formula shows that vorticity is transported by the flow, which is extremely useful in the vortex filaments problem, see \cite{MaB} for references. Yudovich \cite{Yu2} extended the uniqueness results to the solutions whose $L^p$ norm grows slowly as $p \to \infty$ and Vishik \cite{Vis} extended the results to Besov type spaces. The uniqueness of weak solutions to the two-dimensional Euler equation can also be derived via a Lagrangian method, see Crippa and Stefani \cite{St} or Marchioro and Pulvirenti \cite{MaP}.  The existence of global solutions to the two-dimensional Euler equation with $L^p$ initial vorticity has been established by Majda and Diperna in \cite{Ma2}. Later, Delort proved the global existence for measure valued initial data \cite{Del}.  However, the uniqueness might not hold for $L^p$ initial data, we refer to the works by Vishik \cite{Vi1} and \cite{Vi2}.

    \quad Some three-dimensional flows with special symmetries can be reduced to two-dimensional flows. Among the most important examples are axisymmetric flows and helical flows without swirl. The first well-posedness result for axisymmetric solutions was obtained by Yudovich in \cite{Yu3} and then extended to Lorentz spaces (including $L^1 \bigcap L^{\infty}(\R^3)$) by Danchin in \cite{Da}. The global well-posedness for smooth solutions has been established in \cite{GZ}, \cite{Ma1} and \cite{SR}. In \cite{CK}, D. Chae and N. Kim obtained the global existence of weak solutions for $\frac{w_0^{\theta}}{r}\in L^1\bigcap L^p(\R^3)$ with $p>\frac{6}{5}$ and then Jiu, Liu and Niu extended the results to any $p>1$ \cite{JLN2}.
For helical solutions, in \cite{Dut} Dutrifoy proved the global well-posedness for smooth solutions in bounded domains, a key observation is that the third component $\Omega^z$ of the vorticity $\Omega$  satisfies
\begin{equation}\label{eq 3euler z}
\partial_t \Omega^z+U \cdot \nabla \Omega^z=0,
\end{equation}
which shows that $\Omega^z$ is transported by the flow.
Setting $w(x_1,x_2)=\Omega^z(x_1,x_2,0)$, Ettinger and Titi reduced the three-dimensional Euler equation to the following two-dimensional problem:
    \begin{equation*}
\partial_t w+\nabla^{\perp}L_H^{-1}w \cdot \nabla w=0,
\end{equation*}

     where $L_H$ is a specific elliptic operator. They established the global existence and uniqueness of this two-dimensional problem in bounded domains \cite{ET}. Some of their results can be extended to the whole space $\R^3$. The global existence has been proved by Bronzi–Lopes–Lopes \cite{BLN} for $L^1\bigcap L^p(\R^3)$ initial vorticity with compact support. The assumption of the compact support was then removed by Jiu, Li and Niu \cite{JLN}. However, the uniqueness of the solutions, continuous dependence on initial data and the vortex transport formula is still unknown in the whole space $\R^3$. Moreover, the Biot-Savart law in their work is given by
     $$BS[\Omega]=\nabla \wedge (-\Delta)_{\R^2\times \T}^{-1}\Omega+A,$$
     where the constant $A:=\Big(0,0,\frac{1}{4\pi^2}\int_{\R^2\times \T}\Omega^z(y)\,dy\Big)$ is used to ensuring that $BS[\Omega]$ has vanishing swirl.

       \quad In this paper, we will always use the Biot-Savart law\footnote{Theorem \ref{thm main}, Corollary \ref{co main} and Theorem \ref{thm conservation} remain valid if we set $$U(x)=-\frac{1}{4\pi}\int_{\R^3}\frac{x-y}{|x-y|^3}\wedge \Omega(y)\,dy+(0,0,A),$$ for any $A \in \R$.} given by \eqref{eq biot savart law}
     $$U(x)=-\frac{1}{4\pi}\int_{\R^3}\frac{x-y}{|x-y|^3}\wedge \Omega(y)\,dy,$$
     which decays at infinity when $\Omega \in C_c^{\infty}(\R^2\times \T)$.
Note that   $$U(x)=BS[\Omega](x)-A,$$ so the velocity field in our work does not has vanishing swirl. However, the main purpose of this paper is to establish sufficient technical tools for estimating the long-time dynamics of helical vortex filaments \cite{GZ1}. Therefore, we must assume the Biot-Savart law \eqref{eq biot savart law} and we do not concern with the helical swirl.
Now assume $\Omega^z$ is a helical solution to \eqref{eq 3euler z} in $\R^3$, then $\Omega:=(x_2,-x_1,1)\Omega^z$ satisfies the three-dimensional Euler equation. Since $\Omega^z$ is helical, it suffices to consider $w(x_1,x_2,t):=\Omega^z(x_1,x_2,0,t)$, which satisfies the two-dimensional helical Euler equation
\begin{equation}\label{eq 2euler}
\partial_t w +Hw\cdot \nabla w=0,
\end{equation}
where
\begin{equation*}
Hw(x)=\int_{\R^2} H(x,y)w(y)\,dy
\end{equation*}
and $H(x,y)$ is the modified Biot-Savart kernel defined in Proposition \ref{prop 3 euler}. We refer to the two-dimensional problem \eqref{eq 2euler} as the two-dimensional helical Euler equation and we refer to the equation \eqref{eq 3euler z} as the three-dimensional helical Euler equation.

  \quad The main result are stated as follows (for the
precise definition of Lagrangian weak solutions and the function spaces, we refer the reader to Section 2).
\begin{theorem}\label{thm main}
The two-dimensional helical Euler equation \eqref{eq 2euler} is globally well-posed in $L^1_1\bigcap L^{\infty}_1(\R^2)$ in the sense that
\begin{enumerate}
\item (Existence). For all $T>0$ and $w_0 \in L^1 \bigcap L^{\infty}(\R^2)$, there exists a Lagrangian weak solution $w(x,t)\in L^{\infty}\left([0,T],L^1 \bigcap L^{\infty}(\R^2)\right)$ to \eqref{eq 2euler}. Moreover, the velocity $Hw(\cdot,t)$ is locally Log-Lipschitz continuous.  \\
\item (Uniqueness). For any $w_0 \in L^1_1 \bigcap L^{\infty}_1(\R^2)$, there exists at most one lagrangian weak solution to \eqref{eq 2euler} in $L^{\infty}\left([0,T],L^1_1 \bigcap L^{\infty}_1(\R^2)\right)$ with initial vorticity $w_0$.\\
    \item (Continuous dependence on initial data). Let $w_{0,n}$ be a sequence of initial data such that
\begin{equation*}
\sup_n \|w_{0,n}\|_{L^{\infty}_1(\R^2)} < \infty
\end{equation*}
and
\begin{equation*}
\|w_{0,n}-w_0\|_{L^1_1(\R^2)} \to 0.
\end{equation*}
Let $w(t)$ and $w_n(t)$ be the solution to \eqref{eq 2euler} with initial data $w_0$ and $w_{0,n}$ respectively, then for any $T>0$, there holds
\begin{equation*}
\sup_{0\le t\le T} \|w_n(t)-w(t)\|_{L^1_1(\R^2)} \to 0.
\end{equation*}.
\end{enumerate}
\end{theorem}
\emph{Comments on the results:}
\begin{enumerate}
    \item   The assumption of Lagrangian in (ii) can be removed. Since the velocity field $Hw$ is divergence free, Lemma \ref{le uniqueness of the transport equation} guarantees that all weak solutions to \eqref{eq 2euler} in $L^{\infty}\left([0,T],L^1_1 \bigcap L^{\infty}_1(\R^2)\right)$ are indeed Lagrangian.
   \item The solution $w$ to \eqref{eq 2euler} with $w_0 \in L^1_1 \bigcap L^{\infty}_1(\R^2)$ belongs to $C\left([0,T],L^1_1(\R^2)\right)$. This can be proved easily with the help of the vortex transport formula.
   \item  The global well-posedness for \eqref{eq 2euler} in $L^1_m \bigcap L^{\infty}_m(\R^2)$ can be obtained for any $m \ge 1$ by a similar argument. But the global existence of weak solutions to \eqref{eq 2euler} can be established only when $m=0$. Indeed, the continuous dependence on initial data still holds by the same argument in Section 6 once the uniqueness is obtained. However, our method could not be applied to prove the uniqueness when $m=0$. The main difficulty is that $Hw(x)$ is only locally Log-Lipschitz continuous when $m=0$ instead of globally Log-Lipschitz continuous when $m=1$. So we could not get the desired bound on our distance function $D(t)$. For the locally Log-Lipschitz continuous velocity field, Clop, Jylh$\ddot{a}$, Mateu, and Orobitg proved the uniqueness of continuity equations. We refer the reader to \cite{Clop} for references.
   \end{enumerate}

   The well-posedness for \eqref{eq 3euler z} can also be established by a similar proof to that of Theorem \ref{thm main}. Indeed, the proof for \eqref{eq 3euler z} is even easier since the velocity remains bounded in \eqref{eq 3euler z} while the velocity may grow linearly in \eqref{eq 2euler}.

\begin{corollary}\label{co main}
The three-dimensional helical Euler equation \eqref{eq 3euler z} is globally well-posed in $L^1_1\bigcap L^{\infty}_1(\R^3)$.
\end{corollary}
By a direct calculation, we deduce the conservation of the second momentum $M_1(t)$ and the energy $E(t)$. The conservation of $M_1(t)$ will lead directly to the existence of straight vortex filament in helical setting, a result that was previously only known in the context of Euler's equation with planar symmetry. Furthermore, the energy $E(t)$ will be used to control the velocity for the helical vortex filament, see \cite{GZ1} for a detailed description.

\begin{theorem}\label{thm conservation}Let $\Omega^z \in L^1_2\bigcap L^{\infty}_2(\R^2\times \T)$ be a helical solution to \eqref{eq 3euler z}. Then the second momentum
$$
M_1(t):=\int_{\R^2\times\T} (x_1^2+x_2^2)\Omega^z(x,t)\,dx
$$
and the (pseudo) energy
$$
E(t):=\int_{\R^2\times \T} (-\Delta_{\R^2\times \T})^{-1}\Omega \cdot \Omega \,dx
$$
are conserved, where $\Omega(x,t):=(x_2,-x_1,1)\Omega^z(x,t)$ and the operator $(-\Delta_{\R^2\times \T})^{-1}$ is defined in next section.

\end{theorem}

  \quad Now we state our strategy of the proof. First, we show that every weak solution to \eqref{eq 2euler} in $L^1_1 \bigcap L^{\infty}_1(\R^2)$ can be lifted to a weak solution to \eqref{eq 3euler z} in $L^1_1 \bigcap L^{\infty}_1(\R^2\times \T)$ and such lifting is bijective. So in order to prove Theorem \ref{thm main}, it suffices to show the existence of weak solutions and the continuous dependence on initial data for the two-dimensional helical Euler equation \eqref{eq 2euler} and the uniqueness of weak solutions to \eqref{eq 3euler z} in $L^1_1 \bigcap L^{\infty}_1(\R^2\times \T)$. In our proof, the following observation is very important. Let $X(\alpha_1,\alpha_2,t)=(X_1(\alpha_1,\alpha_2,t),X_2(\alpha_1,\alpha_2,t))$ be the particle trajectory map associated with the two-dimensional modified Biot-Savart kernel $Hw$. Define
\begin{equation*}
X_3(\alpha_1,\alpha_2,t)=\int_0^t U_3(X_1(\alpha_1,\alpha_2,s),X_2(\alpha_1,\alpha_2,s),0)\,ds
\end{equation*}
and
\begin{equation*}
Y(\alpha_1,\alpha_2,0,t)=(R_{X_3}(X_1,X_2),X_3).
\end{equation*}
Then
\begin{equation*}
Y(\alpha_1,\alpha_2,\alpha_3,t):=S_{\alpha_3}Y(R_{-\alpha_3}(\alpha_1,\alpha_2),0,t)
\end{equation*}
is the particle trajectory map associated with the three-dimensional modified Biot-Savart kernel $U$.

    \quad \emph{\bf Global existence of \eqref{eq 2euler}:} First we estimate the two-dimensional modified Biot-Savart kernel $H(x,y)$ and we show that it admits the decomposition
     $$H(x,y)=H_1(x,y)+H_2(x,y)$$
     for $H_2(x,y)$ a purely rotational part satisfying $(x_1,x_2)\cdot H_2(x,y)=0$,
    which has no contribution to the growth of the trajectory. So we may focus our attention on $H_1(x,y)$. By splitting the integral of $H_1(x,y)$ into different regions, we prove that
     $$|H_1(x,y)| \lesssim \frac{\la x \ra}{|x-y|} \quad \text{and} \quad |H_1(x,y)-H_1(z,y)| \lesssim \frac{\la x \ra |x-z|}{|x-y||z-y|}, $$
     which behaves like the two-dimensional Biot-Savart kernel $\frac{(x-y)^{\perp}}{|x-y|^2}$, except for a linear growth at infinity. Next, we use the Picard iteration to prove the global existence of weak solutions for smooth initial data. Let $w_0 \in C_{c}^{\infty}(\R^2)$ and set $w^0(x,t):=w_0(x)$. Then we define $w^n(x,t)$ inductively:
\begin{equation*}
\left\{\begin{aligned}\partial_t w^{n+1}+Hw^n \cdot \nabla w^{n+1}=&0 \\
w^{n+1}(x,0)=&w_0(x). \end{aligned}\right.
\end{equation*}
We remark that following the standard compactness argument in \cite{MaB}, one can find a subsequence $n_k$ such that
$$w_{n_k} \to w, \quad  \text{and} \quad w_{n_k+1} \to \tilde{w} $$
for bounded measurable functions $w$ and $\tilde{w}$ with compact support in $\R^2$. So letting $k \to \infty$, there holds
$$
\partial_t w+ H\tilde{w}\cdot \nabla w=0
$$
and it remains to prove that $w = \tilde{w}$. To this end, we introduce the distance function
$$
D_n(t):=\int_{\R^2}\left| X^n(\alpha,t)-X^{n+1}(\alpha,t) \right||w_0(\alpha)|\,d\alpha.
$$
We will see that $D^*(t):= \limsup_n D_n(t)$ satisfies the integral inequality
$$
D^*(t) \lesssim \int_0^t F(D^*(s))\,ds
$$
for a Log-Lipschitz continuous function $F$. Thus, $X^n$ is a Cauchy sequence in $L^1(|w_0|dx)$, which implies $X=\tilde{X}$ and hence $w=\tilde{w}$.  Now for general initial data $w_0 \in L^1\bigcap L^{\infty}(\R^2)$, we choose a sequence of $w_{0,n} \in C_c^{\infty}$ such that $w_{0,n} \to w_0$ in $L^1$ and $|w_{0,n}(x)|\lesssim 1$. Then letting $n \to \infty$ we find the desired solution to \eqref{eq 2euler}.

      \quad \emph{\bf Uniqueness:} Let $\Omega$ and $\tilde{\Omega}$ be two weak solutions to the three-dimensional helical Euler equation \eqref{eq 3euler z} and let $X(\alpha,t)$, $\tilde{X}(\alpha,t)$ be their particle trajectory maps, respectively. Then motivated by the uniqueness proof of Crippa and Stefani \cite{St}, we introduce the distance function
      $$D(t):=\int_{\R^2\times \T} |X(\alpha,t)-\tilde{X}(\alpha,t)||\Omega_0(\alpha)|\,d\alpha.$$
        We show that the distance function $D(t)$ satisfies an integral inequality
        $$D(t)\lesssim \int_0^t F(D(s)) \,ds$$
        for an Osgood continuous function $F$. Hence the uniqueness follows from Osgood's Lemma.

          \quad \emph{\bf Continuous dependence on initial data:} With the help of the vortex transport formula, this can be proved via a compactness argument.

\quad This paper is aimed to established sufficient technical tools which will be needed in \cite{GZ1}. There are several innovations in our paper. First,  we prove the global well-posedness for \eqref{eq 2euler} in the whole space $\R^2$, the velocity field given by \eqref{eq biot savart law} has constant constant helical swirl $$U_{\xi}(x)\equiv-\frac{1}{4\pi^2}\int_{\R^2\times \T}\Omega^z_0(y)\,dy.$$ This is an extension of the result in \cite{ET}, which proved the global existence and uniqueness for helical Euler equations in bounded domains with vanishing helical swirl.  Second, the vortex transport formula is obtained, which is unknown before even in bounded domains. The vortex transport formula will play a crucial role in the vortex filament problems. Third,  a rigorous justification of the equivalence between the two-dimensional helical Euler equation \eqref{eq 2euler} and the three-dimensional helical Euler equation \eqref{eq 3euler z} is provided. Moreover, the relationship between their particle trajectory maps is also obtained in our article. Finally, we get a detailed estimates for the modified Biot-Savart kernel associated with the helical Euler equation. The estimate does not require any $W^{1,p}$ regularity for the velocity, so the proof we applied here can be easily extended to more generalized Biot-Savart kernels.

\quad Our article is organized as follows: In Section $2$, we give a detailed description of our problems and reduce the three-dimensional incompressible Euler equation \eqref{eq 3euler} to the two-dimensional helical Euler equation \eqref{eq 2euler}. In Section $3$, we obtain detailed estimates for the modified Biot-Savart kernel $H(x,y)$. In Section $4$, we show the uniqueness of the helical Euler equation \eqref{eq 3euler z} and \eqref{eq 2euler} in $L^1_1 \bigcap L^{\infty}_1$. In Section $5$, we obtain the global existence and the vortex transport formula for \eqref{eq 2euler}. In Section $6$, we prove the continuous dependence on initial data to equation \eqref{eq 3euler z} and \eqref{eq 2euler}.

\section{Mathematical preliminaries and main result. }
The main purpose of this section is to fix notations and state our main results. For simplicity of presentation, we usually refer a point $x=(x_1,x_2,x_3) \in \R^3$ or $\R^2\times \T$ to $x=(x',x_3)$.
\subsection{Function spaces}

 In this subsection, we will introduce some functional spaces that will be used later.
The three-dimensional weighted $L^p_m(\R^2 \times \T)$ norm is defined by
\begin{equation*}
\|f\|_{L^p_m(\R^2\times \T)}:=\int_{\R^2\times \T}(1+y_1^2+y_2^2)^{\frac{pm}{2}} |f(y_1,y_2,y_3)|^p\,dy^{\frac{1}{p}}
\end{equation*}
for $1\le p< \infty$ and
$$\|f\|_{L^{\infty}_m(\R^2\times \T)}:=\sup_{x\in \R^3}(1+x_1^2+x_2^2)^{\frac{m}{2}} |f(x_1,x_2,x_3)|.$$
Similarly, the two-dimensional weighted $L^p_m(\R^2)$ norm is defined by
\begin{equation*}
\|g\|_{L^p_m(\R^2)}:=\int_{\R^2}\la y \ra^{pm} |g(y_1,y_2)|^p\,dy^{\frac{1}{p}}
\end{equation*}
for $1\le p< \infty$ and
$$\|f\|_{L^{\infty}_m(\R^2)}:=\sup_{x\in \R^2}\la x \ra^{m} |f(x_1,x_2)|.$$
Here and throughout the paper we use the notation $\la a\ra=(1+|a|^2)^{\frac12}$.
The norms $L^p_m(\R^2 \times \T)$ and $L^p_m(\R^2)$ are related as follows:
\begin{lemma}\label{le norm of R^2 and R^3}
Let $f \in L^p_m(\R^2)$ and define $F: \R^3 \to \R$ as
\begin{equation*}
F(x',x_3):=f(R_{-x_3}x').
\end{equation*}
Then for any $p \in [1,+\infty]$,
\begin{equation*}
\|f\|_{L^p_m(\R^2)}=(2\pi)^{-1/p}\|F\|_{L^p_m(\R^2 \times \mathbb{T})}.
\end{equation*}
\end{lemma}
\begin{proof}
The proof follows directly by definition when $p=\infty$, so it suffices to consider the case when $1 \le p <\infty$. In fact,
\begin{equation*}\begin{aligned}
\|F\|_{L^p_m(\R^2 \times \mathbb{T})}^p &= \int_{\mathbb{T}} \int_{\R^2} (1+y_1^2+y_2^2)^{\frac{mp}{2}} |f(R_{-y_3}(y_1,y_2))|^p\,dy_1dy_2 \,dy_3 \\
&= \int_{\mathbb{T}} \int_{\R^2} (1+z_1^2+z_2^2)^{\frac{mp}{2}} |f(z_1,z_2)|^p\,dz_1dz_2 \,dy_3 \\
&=2\pi \|f\|_{L^p_m(\R^2)}^p.
\end{aligned}\end{equation*}

\end{proof}

\subsection{Helical functions and vector fields. }

\begin{definition}[helical function]\label{def helical function}
A function $f: \R^3 \to \R$ is called helical if for all $\theta \in \R$,
\begin{equation*}
f(S_{\theta}x)=f(x)
\end{equation*}
for almost every $x\in \R^3$, where
\begin{equation*}
S_{\theta}x:=R_{\theta}x+\begin{pmatrix} 0 \\ 0 \\ \theta \end{pmatrix}
\end{equation*}
and
\begin{equation}\label{eq def of R theta}
R_{\theta}:= \begin{pmatrix}  \cos{\theta} & \sin{\theta} & 0 \\ -\sin{\theta} & \cos{\theta} & 0 \\ 0 & 0 &1 \end{pmatrix}.
\end{equation}
\end{definition}

\begin{definition}[helical vector field]\label{def helical vector field}
A vector field $u: \R^3 \to \R^3$ is called helical if for all $\theta \in \R$,
\begin{equation*}
R_{-\theta}u(S_{\theta}(x))=u(x)
\end{equation*}
for almost every $x \in \R^3$.

\end{definition}
 Assume $f$ is a continuous helical function and $u$ is a continuous helical vector field. With a slight abuse of notation, we set $R_{\theta}x=\begin{pmatrix}  \cos{\theta} & \sin{\theta}  \\ -\sin{\theta} & \cos{\theta} \end{pmatrix} \begin{pmatrix} x_1 \\ x_2 \end{pmatrix}.$ Then by definitions above we see that
\begin{equation}\label{eq helical function}
f(x_1,x_2,x_3)=f(R_{-x_3}(x_1,x_2),0)
\end{equation}
and
\begin{equation}\label{eq helical vector}
u(x_1,x_2,x_3)=R_{x_3}u(R_{-x_3}(x_1,x_2),0).
\end{equation}

Now if $f$ and $u$ are only measurable function (vector field), then $f(x',0)$ and $u(x',0)$ are not well-defined, thus \eqref{eq helical function} and \eqref{eq helical vector} do not make sense. However, we have the following:
\begin{lemma}
Let $f$ be a locally bounded helical function and $u$ be a locally bounded helical vector field. Define
$$
\tilde{f}(x')=\frac{1}{2\pi}\int_0^{2\pi} f(R_{a}x',a) \,da
$$
and
$$
\tilde{u}(x')=\frac{1}{2\pi}\int_0^{2\pi} R_{-a}u(R_{a}x',a) \,da.
$$
Then for almost every $x\in \R^3$, there hold
\begin{equation}\label{eq f(x1,x2,0)}
f(x',x_3)=\tilde{f}(R_{-x_3}x')
\end{equation}
and
\begin{equation}\label{eq u(x1,x2,0)}
u(x',x_3)=R_{x_3}\tilde{u}(R_{-x_3}x').
\end{equation}

\end{lemma}
$$$$

\begin{proof}We only prove \eqref{eq f(x1,x2,0)} since the proof of \eqref{eq u(x1,x2,0)} is similar.
Setting
$$
A_r=\int_{B_r} \Big| f(x',x_3)-g(R_{-x_3}x')  \Big|\,dx'dx_3,
$$
then it suffice to show that $A_r=0$ for all $r>0$. By the definition of $g$, we obtain
\begin{equation*}\begin{aligned}
A_r &= \int_{B_r} \left| f(x',x_3)-\frac{1}{2\pi}\int_0^{2\pi} f(R_{a-x_3}x',a)\,da  \right| \,dx \\
&\le \frac{1}{2\pi}\int_{B_r}  \int_0^{2\pi}\left|f(x',x_3)- f(R_{a-x_3}x',a)\right|\,da   \,dx.
\end{aligned}\end{equation*}
Next, we make a change of variable $a=b+x_3$, then it follows that
\begin{equation*}\begin{aligned}\label{eq sec2 A1}
A_r&\le\frac{1}{2\pi}\int_{B_r}  \int_{-x_3}^{2\pi-x_3}\left|f(x',x_3)- f(R_{b}x',b+x_3)\right|\,db   \,dx \\
&\le\frac{1}{2\pi}\int_{B_r}  \int_{-r}^{2\pi+r}\left|f(x',x_3)- f(R_{b}x',b+x_3)\right|\,db   \,dx \\
&=\frac{1}{2\pi}  \int_{-r}^{2\pi+r}\int_{B_r}\left|f(x',x_3)- f(S_{b}(x',x_3))\right|\,dx   \,db.
\end{aligned}\end{equation*}
Thus, by Definition \ref{def helical function} we get $A_r \equiv 0$ and hence \eqref{eq f(x1,x2,0)} holds.

\end{proof}
Motivated by \eqref{eq helical function}, \eqref{eq helical vector}, by slight abuse of notations, we may refer to $\tilde{f}(x')$ as $f(x',0)$, and refer to $\tilde{u}(x')$ as $u(x',0)$.
 Next, we derive some useful properties for helical functions and vector fields.

\begin{lemma}\label{le partial3 scalar}
Let $f \in L^{\infty}_{loc}(\R^3)$ be a helical function, then
\begin{equation}\label{eq partial3 scalar}
\partial_3f=x_1\partial_2f-x_2\partial_1f
\end{equation}
in the sense of distribution.

\begin{proof}

In \cite{ET}, \eqref{eq partial3 scalar} has been shown when $f$ is smooth. For general $f \in L^{\infty}_{loc}(\R^3)$, we define $f_0(x_1,x_2):=f(x_1,x_2,0)$. Let $f_{0,n} \in C_{c}^{\infty}(\R^2)$ be a sequence of smooth function converges to $f_0$ in $L^1_{loc}(\R^2)$ and set $f_n(x',x_3):=f_{0,n}(R_{-x_3}x')$, then for any $\phi \in C^{\infty}_c(\R^3)$, it follows from \eqref{eq def of R theta} and Definition \ref{def helical function} that
\[\begin{aligned}
\int_{\R^3} f(x) \partial_3 \phi(x) \,dx&=\int_{\R^3} f_0(R_{-x_3}x')\partial_3\phi(x)\,dx \\
&=  \int_{\R^3} f_0(x_1\cos{x_3}-x_2\sin{x_3},x_1\sin{x_3}+x_2\cos{x_3}) \partial_3\phi(x)\,dx.
\end{aligned}\]
Since $\phi$ has compact support, the dominating convergence theorem gives
\[\begin{aligned}\int_{\R^3} f(x) \partial_3 \phi(x) \,dx
&= \lim_n \int_{\R^3} f_{0,n}(x_1\cos{x_3}-x_2\sin{x_3},x_1\sin{x_3}+x_2\cos{x_3})\partial_3\phi(x)\,dx \\
&= \lim_n \int_{\R^3} f_n(x)\partial_3 \phi(x)\,dx.
\end{aligned}\]
Note that \eqref{eq partial3 scalar} holds for smooth helical function $f_n$, so integration by parts yields
\[\begin{aligned}\int_{\R^3} f(x) \partial_3 \phi(x) \,dx
&=\lim_n \int_{\R^3} -\partial_3f_n(x) \phi(x)\,dx \\
&=\lim_n \int_{\R^3} (x_2 \partial_1f_n -x_1\partial_2f_n)\phi(x)\,dx \\
&=\lim_n \int_{\R^3} f_n(x_1\partial_2\phi(x)-x_2\partial_1\phi(x)) \,dx \\
&=\lim_n \int_{\R^3} (x_1\partial_2\phi(x)-x_2\partial_1\phi(x))f_{0,n}(x_1\cos{x_3}-x_2\sin{x_3},x_1\sin{x_3}+x_2\cos{x_3})   \,dx.
\end{aligned}\]
Again by dominating convergence theorem, we finally get
\[\begin{aligned}\int_{\R^3} f(x) \partial_3 \phi(x) \,dx
&=\int_{\R^3} (x_1\partial_2\phi(x)-x_2\partial_1\phi(x))f_0(x_1\cos{x_3}-x_2\sin{x_3},x_1\sin{x_3}+x_2\cos{x_3})  \,dx \\
&=\int_{\R^3} f(x) \left(x_1\partial_2\phi(x)-x_2\partial_1\phi(x)\right) \,dx,
\end{aligned}\]
which proves \eqref{eq partial3 scalar}.
\end{proof}
\end{lemma}

For helical vector fields, by \cite{ET} and a similar argument as above, we obtain

\begin{lemma}\label{le helical vector}
Let $u=(u_1,u_2,u_3) \in L^{1}_{loc}$ be a helical vector field, then we have
\begin{equation*}
\partial_3u_1=x_1\partial_2u_1-x_2\partial_1u_1+u_2,
\end{equation*}
\begin{equation*}
\partial_3u_2=x_1\partial_2u_2-x_2\partial_1u_2-u_1,
\end{equation*}
and
\begin{equation}\label{eq sec2 A0}
\partial_3u_3=x_1\partial_2u_3-x_2\partial_1u_3.
\end{equation}
\end{lemma}
\begin{proof}
Assuming $u$ is smooth, then the above equations have already been proved in \cite{ET}. For general $u \in L^1_{loc}$, the proof is only a matter of smoothness as in Lemma \ref{le partial3 scalar} so we omit it.
\end{proof}
Next, we will describe the particle trajectory map associated with Log-Lipschitz continuous helical vector fields.
\begin{definition}
A function (vector field) $ u$ is called Log-Lipschitz continuous if
\begin{equation*}
\sup_{x\neq y}\frac{\left| u(x)-u(y) \right|}{F(|x-y|)} <\infty,
\end{equation*}
where $F$ is the Log-Lipschitz function
\begin{equation}\label{eq def of F}F(r)=\begin{cases}
r(1-\log r) \quad r \le \frac{1}{e} \\
r+\frac{1}{e} \quad r>\frac{1}{e}.
\end{cases}\end{equation}
\end{definition}
Note that $F'(r)=-\log r$ if $r \le \frac{1}{e}$ and $F'(r)=1$ if $r>\frac{1}{e}$, so $F'(r)$ is continuous and decreasing. Thus, $F(r)$ is a concave function which satisfies $F(r) \approx r(1-\log^-r)$, where
\begin{equation*}\log^-r=\begin{cases}
\log r \quad r \le 1 \\
0 \quad r>1.
\end{cases}\end{equation*}
\begin{definition}
Let $u: \R^d \times [0,T) \to \R^d$ be a locally Log-Lipschitz continuous velocity field, we say $X(\alpha,t)$ is a particle trajectory map associated with $u$ if it satisfies
\begin{equation*}\begin{cases}\begin{aligned}
\frac{dX(\alpha,t)}{dt}&= u(X(\alpha,t),t)\\
X(\alpha,0)&=\alpha
\end{aligned}\end{cases}\end{equation*}
for any $\alpha \in \R^d$ and $t \in [0,T)$.
\end{definition}
$$$$
The particle trajectory map associated with a helical vector field satisfies the following:
\begin{lemma}\label{le helical particle trajectory map}
Let $u$ be a locally Log-Lipschitz continuous helical vector field and $X(\alpha,t)$ be its associated particle trajectory map, then for any  $\alpha \in \R^3$, $t \in \R^+$ and $\theta \in \R$,
\begin{equation}\label{eq helical particle map}
S_{-\theta}X(S_{\theta}\alpha,t)=X(\alpha,t).
\end{equation}
\end{lemma}
\begin{proof}
A direct calculation shows that
\begin{equation*}\begin{aligned}
\frac{d\left(S_{-\theta}X(S_{\theta}\alpha,t)\right)}{dt}&=\frac{d\left(R_{-\theta}X(S_{\theta}\alpha,t)\right)}{dt} \\
&=R_{-\theta}u(X(S_{\theta}\alpha,t),t) \\
&=u(S_{-\theta}X(S_{\theta}\alpha,t),t).
\end{aligned}\end{equation*}
Thus, both $S_{-\theta}X(S_{\theta}\alpha,t)$ and $X(\alpha,t)$ are the solutions of
\begin{equation*}\begin{cases}\begin{aligned}
\frac{dX(\alpha,t)}{dt}&= u(X(\alpha,t),t)\\
X(\alpha,0)&=\alpha.
\end{aligned}\end{cases}\end{equation*}

So the uniqueness of the particle trajectory map implies
\begin{equation*}
S_{\theta}X(S_{-\theta}\alpha,t)=X(\alpha,t).
\end{equation*}
\end{proof}

\subsection{Biot-Savart law} The main purpose of this subsection is to the following:
\begin{equation}\label{eq sec2 C1}
\nabla\,(-\Delta_{\R^2\times \R})^{-1} f=-\frac{1}{4\pi} \int_{\R^3}\frac{x-y}{|x-y|^3}f(y)\,dy,
\end{equation}
which will subsequently be used to prove the conservation of the energy. First we introduce the operator $\Delta_{\R^2\times\T}^{-1}$.
\begin{proposition}[\cite{BLN}]\label{prop stream function}
The Green's function for the Laplacian in $\R^3$ with $2\pi-$periodic boundary condition in the $x_3-$direction is given by
\begin{equation}\label{eq sec2 C0}
\mathbf{G}(x)=\frac{1}{4\pi^2}\log\frac{1}{|x'|}+\frac{1}{2\pi^2}\sum_{n=1}^{\infty}K_0(n|x'|)\cos(nx_3)
\end{equation}
for all $x \in \R^2\times \T$ with $\tilde{x}\neq 0$, where
\begin{equation}
K_0(z):=\int_0^{\infty}\frac{\cos(t)}{\sqrt{z^2+t^2}}\,dt.
\end{equation}
More precisely, the stream function $$\psi(x):=\int_{\R^2\times \T} \mathbf{G}(x-y)f(y)\,dy$$ is a periodic function in $x_3$ and satisfies $$-\Delta \psi=f.$$
\end{proposition}

\begin{lemma}
For any $f,g \in L^1_1 \bigcap L^{\infty}_1(\R^2\times \T)$,
$$
\int_{\R^2\times \T}\int_{\R^2\times \T} \mathbf{G}(x-y)f(x)g(y)\,dx\,dy
$$
converges absolutely. In particular, one might freely interchange the sum and the integral in above expression.
\end{lemma}
\begin{proof}
First we note that it suffices to show
\begin{equation}\label{eq sec3 D0}
|K_0(z)|\lesssim z^{-\frac{3}{2}}.
\end{equation}
Assuming this holds true, then
\begin{equation*}\begin{aligned}
|\mathbf{G}(x-y)| \lesssim& \left| \log |x'-y'|\right| + \sum_{n=1}^{\infty} \frac{1}{n^{3/2}|x'-y'|^{3/2}} \\
\lesssim& 1+|x'-y'| +\frac{1}{|x'-y'|^{3/2}}.
\end{aligned}\end{equation*}
For the first term and the second term, we have
$$
\int_{\R^2\times \T}\int_{\R^2\times \T}(1+|x'-y'|)|f(x)g(y)|\,dx\,dy \lesssim \|f\|_{L^1_1\bigcap L^{\infty}_1}\|g\|_{L^1_1\bigcap L^{\infty}_1}.
$$
For the last term,
\begin{equation*}\begin{aligned}
&\int_{\R^2\times \T}\int_{\R^2\times \T}  \frac{1}{|x'-y'|^{3/2}}|f(x)g(y)| \,dx\,dy \\
\lesssim& \int_{|x'-y'|\ge 1}  \frac{1}{|x'-y'|^{3/2}}|f(x)g(y)| \,dx\,dy +\int_{|x'-y'|\le 1}   \frac{1}{|x'-y'|^{3/2}}|f(x)g(y)| \,dx\,dy \\
\lesssim& \int_{\R^2\times \T}\int_{\R^2\times \T} |f(x)g(y)| \,dx\,dy+\int_{\R^2\times \T}\int_{|y'|\le 1} \frac{\|g\|_{L^{\infty}}}{|y'|^{3/2}}\,dy |f(x)|\,dx \\
\lesssim& \|f\|_{L^1}\|g\|_{L^1}+\|f\|_{L^1}\|g\|_{L^{\infty}}\\
\lesssim& \|f\|_{L^1_1\bigcap L^{\infty}_1}\|g\|_{L^1_1\bigcap L^{\infty}_1}.
\end{aligned}\end{equation*}
Next it remains to prove \eqref{eq sec3 D0}. On one hand,
\begin{equation*}\begin{aligned}
K_0(z)=& \frac{1}{z}\int_0^{\infty} (1+t^2)^{-1/2} \,d \sin (zt) \\
=& \frac{1}{z}\int_0^{\infty} t\sin(zt)(1+t^2)^{-3/2} \,dt,
\end{aligned}\end{equation*}
which implies that $|K_0(z)|\lesssim \frac{1}{|z|}.$ On the other hand,
\begin{equation*}\begin{aligned}
K_0(z)=&\frac{1}{z}\int_0^{\infty} t\sin(zt)(1+t^2)^{-3/2} \,dt, \\
=& \frac{1}{z^2} \int_0^{\infty} \cos(zt)\left( \frac{1}{(1+t^2)^{3/2}}-\frac{3t^2}{(1+t^2)^{5/2}}  \right)\,dt,
\end{aligned}\end{equation*}
which shows $|K_0(z)| \lesssim \frac{1}{z^2}$ and hence we get \eqref{eq sec3 D0} by interpolation.
\end{proof}
To prove \eqref{eq sec2 C1}, one might not simply differentiate \eqref{eq sec2 C0} term by term since the integral of the Fourier series might not converges absolutely. Therefore, a careful argument is required. On one hand, we have

\begin{equation}\label{eq sec2 C2}
\begin{aligned}
&-\frac{1}{4\pi}\int_{\R^3}\frac{x-y}{|x-y|^3}f(y)\,dy\\
=&-\frac{1}{4\pi}\int_{\R^2\times \T} \sum_{n=-\infty}^{\infty} \frac{(x_1-y_1,x_2-y_2,x_3-y_3-2n\pi)}{(|x'-y'|^2+(x_3-y_3-2n\pi)^2)^{\frac{3}{2}}}f(y)\,dy,
\end{aligned}
\end{equation}
which converges absolutely due to the following lemma.
\begin{lemma}\label{le bound for u in R^3}Let $\Omega(x_1,x_2,x_3)$ be a $2\pi$-periodic function in $x_3$, then one has
\begin{equation}\label{eq le bound for u}
\int_{\R^3} \frac{1}{|x-y|^2} |\Omega(y)|\,dy \lesssim \|\Omega\|_{L^1\bigcap L^{\infty}(\R^2 \times \mathbb{T})}.
\end{equation}
\end{lemma}
\begin{proof}
First, observe that
\begin{equation*}
\int_{\R^3} \frac{1}{|x-y|^2} |\Omega(y)|\,dy=\int_{\R^3} \frac{1}{|y|^2} |\Omega(y+x)|\,dy
\end{equation*}
and for any $x\in \R^3$
\begin{equation*}
\|\Omega(\cdot+x)\|_{L^1\bigcap L^{\infty}(\R^2 \times \mathbb{T})}=\|\Omega\|_{L^1\bigcap L^{\infty}(\R^2 \times \mathbb{T})}.
\end{equation*}
So we may assume without loss of generality that $x=0$ and $\Omega \ge 0$. Then we estimate
\begin{equation}\label{eq sec3 C0}\begin{aligned}
\int_{\R^3} \frac{\Omega(y)}{|y|^2}\,dy &=\int_{\R^2}\int_{-2\pi}^{2\pi} \frac{\Omega(y)}{|y|^2}\,dy_3\,dy_1dy_2+\int_{\R^2}\int_{|y_3|\ge 2\pi}\frac{\Omega(y)}{|y|^2}\,dy.
\end{aligned}\end{equation}
For the first term in \eqref{eq sec3 C0}, denote
$$
\tilde{\Omega}(y):=\Omega(y)\chi_{|y_3|\le 2\pi} (y).
$$
Recall that
\begin{equation}\label{eq standard R3}\begin{aligned}
\int_{\R^3}\frac{|f(y)|}{|x-y|^2}\,dy
&\lesssim \|f\|_{L^1 \bigcap L^{\infty}(\R^3)},
\end{aligned}\end{equation}
so we have
\begin{equation*}
\int_{\R^2}\int_{-2\pi}^{2\pi} \frac{\Omega(y)}{|y|^2}\,dy_3\,dy_1dy_2 \lesssim \|\tilde{\Omega}\|_{L^1\bigcap L^{\infty}(\R^3)}\le 2\|\Omega\|_{L^1\bigcap L^{\infty}(\R^2 \times \mathbb{T})}.
\end{equation*}

For the second term in \eqref{eq sec3 C0}, a direct calculation shows that
\begin{equation*}\begin{aligned}
\int_{\R^2}\int_{|y_3|\ge 2\pi}\frac{\Omega(y)}{|y|^2}\,dy &= \sum_{n \neq -1,0} \int_{\R^2} \int_{2\pi n}^{2\pi(n+1)} \frac{\Omega(y)}{|y|^2}\,dy_3dy_1dy_2 \\
&\lesssim \sum_{n \neq -1,0} \int_{\R^2} \int_{2\pi n}^{2\pi(n+1)} \frac{\Omega(y)}{|y_3|^2}\,dy_3dy_1dy_2 \\
&\lesssim \sum_{n \neq -1,0} \int_{\R^2} \int_{2\pi n}^{2\pi(n+1)} \frac{\Omega(y)}{n^2}\,dy_3dy_1dy_2 \\
&=\sum_{n\neq -1,0} \frac{1}{n^2} \|\Omega\|_{L^1(\R^2\times \T)} \\
&\lesssim \|\Omega\|_{L^1(\R^2 \times \T)},
\end{aligned}\end{equation*}
which completes the proof.
\end{proof}
On the other hand, it follows from \cite{BLN} that the series in \eqref{eq sec2 C0} can be rewritten in Schloeminch series,
\begin{equation}\begin{aligned}
\sum_{n=1}^{\infty}K_{0}(n|x'|)\cos(nx_3)&=\frac{1}{2}\left(\gamma+\log\left(\frac{1}{4\pi}\right)+\log(|x'|)  \right)+\frac{\pi}{2|x|}\\
+&\frac{\pi}{2}\sum_{m=1}^{\infty}\left[ \frac{1}{\sqrt{|x'|^2+(x_3-2m\pi)^2}}-\frac{1}{2m\pi} \right] \\
+&\frac{\pi}{2}\sum_{m=1}^{\infty}\left[ \frac{1}{\sqrt{|x'|^2+(x_3+2m\pi)^2}}-\frac{1}{2m\pi} \right].
\end{aligned}\end{equation}
It has been shown that the sum above and its derivative is absolutely converges and bounded by $\langle x \rangle$. Therefore,
$$
\nabla (-\Delta_{\R^2\times \T})^{-1}f=-\frac{1}{4\pi}\int_{\R^2\times \T} \sum_{n=-\infty}^{\infty} \frac{(x_1-y_1,x_2-y_2,x_3-y_3-2n\pi)}{(|x'-y'|^2+(x_3-y_3-2n\pi)^2)^{\frac{3}{2}}}f(y)\,dy.
$$
Together with \eqref{eq sec2 C2}, we obtain the following.
\begin{lemma}
Equation \eqref{eq sec2 C1} holds for any $f\in L^1_1\bigcap L^{\infty}_1(\R^2\times \T)$. Moreover, the gradient of the Green's function is anti-symmetric,
\begin{equation}\label{eq sec2 C3}
\nabla \mathbf{G}(-x)=-\nabla \mathbf{G}(x)
\end{equation}
and
\begin{equation}\label{eq sec2 C4}
x_1\partial_2\mathbf{G}(x)-x_2\partial_1 \mathbf{G}(x)=0.
\end{equation}
\end{lemma}

It then follows from Lemma $3.3$ of \cite{JLN} that
$$\nabla \wedge (-\Delta_{\R^2\times \T})^{-1} \Omega (x)\cdot \xi(x) \equiv -\frac{1}{4\pi^2}\int_{\R^2\times \T} \Omega^z(y)\,dy,$$
so the velocity field in our work does not have vanishing helical swirl, which is different from that in \cite{BLN} and \cite{JLN}.

\subsection{Reduction to the two-dimensional helical Euler equation.}In this subsection, we will reduce (formally) the three-dimensional Euler equation \eqref{eq 3euler} to the two-dimensional helical equation \eqref{eq 2euler}.
\begin{lemma}\label{le helical euler absence of stretching}
Let $\Omega^z$ be a helical solution to the three-dimensional helical Euler equation \eqref{eq 3euler z} and define $\Omega(x,t)=\xi(x)\Omega^z(x,t)$, then $$U(x)=-\frac{1}{4\pi}\int_{\R^3}\frac{x-y}{|x-y|^3}\wedge \Omega(y)\,dy$$
is a helical vector field and $\Omega$ satisfies the three-dimensional Euler equation \eqref{eq 3euler}.
\end{lemma}
\begin{proof}
The fact that $U$ defined above is a helical vector field can be checked directly. We only need to verify that \begin{equation}\label{eq sec2 B3}\partial_t \Omega +U\cdot \nabla \Omega=\Omega \cdot \nabla U.\end{equation}
For the right hand side, we apply Lemma \ref{le helical vector} to conclude that
\begin{equation*}\Omega \cdot \nabla U=\Omega^z \frac{dU}{d\xi}=(U_2,-U_1,0)\Omega^z.\end{equation*}
For the left hand side,
\begin{equation*}\begin{aligned}\partial_t x^i\Omega^z +U\cdot \nabla (x^i \Omega^z)&=x^i(\partial_t \Omega^z +U\cdot \nabla \Omega^z)+U_i \Omega^z\\
&=U_i \Omega^z,
\end{aligned}\end{equation*}
which proves \eqref{eq sec2 B3}.
\end{proof}
Since $\Omega^z$ is helical, it suffices to consider the equation for $\Omega^z(x',0,t)$.
\begin{proposition}[Two-dimensional helical Euler equation]\label{prop 3 euler}
Let $\Omega^z$ be a smooth helical solution to the three-dimensional helical Euler equation \eqref{eq 3euler z}. We define the two-dimensional modified Biot-Savart kernel
\begin{equation}\label{eq sec2 B1}
K(x,y)=-\frac{1}{4\pi} \int_{\R} \frac{(x_1,x_2,0)-\left(R_{a}(y),a\right)}{|(x_1,x_2,0)-\left(R_{a}(y),a\right)|^3} \wedge \xi(\left(R_{a}(y),a\right))\,da.
\end{equation}
Set $w(x_1,x_2)=\Omega^z(x_1,x_2,0)$,
\begin{equation}\label{eq sec2 B0}
U(x_1,x_2)=\int_{\R^2} K(x,y)w(y)\,dy
\end{equation}
and
\begin{equation*}
Hw:=H_1w+H_2w=(U^1,U^2)+(-x_2,x_1)U^3.
\end{equation*}
Then $w$ satisfies the two-dimensional helical Euler equation
\begin{equation*}
\partial_t w+Hw\cdot \nabla w=0
\end{equation*}
with $\nabla \cdot Hw=0$.

\end{proposition}
\begin{proof}
Since $\Omega^z$ is a helical solution, it follow from \eqref{eq partial3 scalar} that

\begin{equation*}\label{eq sec2 A2}
\partial_t \Omega^z+(U^1-x_2U^3)\partial_1\Omega^z+(U_2+x_1U^3)\partial_2\Omega^z=0.
\end{equation*}
Let $x_3=0$ in the above equation, then $w(x_1,x_2)$ satisfies the two-dimensional helical Euler equation
\begin{equation*}
\partial_t w+Hw\cdot \nabla w=0,
\end{equation*}
where
\begin{equation}\label{eq sec2 B2}
Hw(x_1,x_2)=\left(U^1(x',0)-x_2U^3(x',0),U^2(x',0)+x_1U^3(x',0)\right).
\end{equation}
Next, we will show that this equation is incompressible. Indeed, the divergence of $Hw$ is
\begin{equation*}
\nabla \cdot Hw=\partial_1U^1-x_2\partial_1U^3+\partial_2U^2+x_1\partial_2U^3.
\end{equation*}
Since $U$ is a helical vector field and $\nabla \cdot U=0$, \eqref{eq sec2 A0} gives
\begin{equation*}
\nabla \cdot Hw=\partial_1U^1+\partial_2U^2+\partial_3U^3=\nabla \cdot U=0.
\end{equation*}
Finally we prove \eqref{eq sec2 B1} and \eqref{eq sec2 B0}. A direct computation shows
\[\begin{aligned}
U(x)&=-\frac{1}{4\pi} \int_{\R^3} \frac{x-y}{|x-y|^3}\wedge (y_2\Omega^z(y),-y_1\Omega^z(y),\Omega^z(y))\,dy \\
&= -\frac{1}{4\pi} \int_{\R^3} \frac{x-y}{|x-y|^3}\wedge \xi(y)w\left(R_{-y_3}(y_1,y_2)\right)\,dy.
\end{aligned}\]
Then change of variables $(z_1,z_2)=R_{-y_3}(y_1,y_2)$ yields
\begin{equation*}\begin{aligned}
U(x',0)&=-\frac{1}{4\pi} \int_{\R^3} \frac{(x',0)-\left(R_{y_3}(z'),y_3\right)}{|(x',0)-\left(R_{y_3}(z'),y_3\right)|^3} \wedge \xi(\left(R_{y_3}z',y_3\right))w(z_1,z_2)\,dz_1dz_2dy_3 \\
&=-\frac{1}{4\pi} \int_{\R^2}\left( \int_{\R} \frac{(x',0)-\left(R_{y_3}(z'),y_3\right)}{|(x',0)-\left(R_{y_3}(z'),y_3\right)|^3} \wedge \xi(\left(R_{y_3}z',y_3\right))\,dy_3\right) w(z_1,z_2)\,dz',
\end{aligned}\end{equation*}
which proves \eqref{eq sec2 B1}  and \eqref{eq sec2 B0}. Similarly, the velocity $Hw$ can be written in the form
\begin{equation}\label{eq def of H}Hw(x)=\int_{\R^2} H(x,y)w(y)\,dy:=\int_{\R^2} H_1(x,y)w(y)\,dy+\int_{\R^2} H_2(x,y)w(y)\,dy\end{equation}
for
\begin{equation}\label{eq def of H1}
H_1(x,y):=\left(  K^1(x,y),K^2(x,y) \right)
\end{equation}

and
\begin{equation}\label{eq def of H2}
H_2(x,y):=(-x_2,x_1)K^3(x,y).
\end{equation}
\end{proof}

We will now formulate the weak problem of the two-dimensional helical Euler equation \eqref{eq 2euler} and the three-dimensional helical Euler equation \eqref{eq 3euler z}.
\begin{definition}
(Weak solutions to the three-dimensional helical Euler equation.) We say that $\Omega^z(x,t)$ is a weak solution to \eqref{eq 3euler z} if
\begin{equation*}
\int_{\R^3} \Omega^z(x,t)\phi(x,t) \,dx-\int_{\R^3} \Omega^z(x,0)\phi(x,0)\,dx=\int_0^t \int_{\R^3} \Omega (\partial_t \phi +U\cdot \nabla \phi) \,dxds.
\end{equation*}
for all $t\in [0,T]$ and for all test function $\phi \in C_c^{\infty}(\R^3\times [0,+\infty))$.
\end{definition}
Similarly, we define
\begin{definition}
(Weak solutions to the two-dimensional helical Euler equation.) We say that $w(x,t)$ is a weak solution to \eqref{eq 2euler} if
\begin{equation*}
\int_{\R^2} w(x,t)\phi(x,t) \,dx-\int_{\R^2} w(x,0)\phi(x,0)\,dx=\int_0^t \int_{\R^2} w (\partial_t \phi +Hw\cdot \nabla \phi) \,dxds.
\end{equation*}
for all $t\in [0,T]$ and for all test function $\phi \in C_c^{\infty}(\R^2\times [0,+\infty))$.
\end{definition}

\begin{definition}
(Lagrangian solution) We say $w(x,t)$ is a Lagrangian solution if it is a weak solution to \eqref{eq 2euler} and satisfies the vortex transport formula
\begin{equation*}
w(X(\alpha,t),t)=w(\alpha,0),
\end{equation*}
where $X(\alpha,t)$ is the particle trajectory map associated with $Hw$.
\end{definition}

\section{Key estimates for the Biot-Savart kernel.}The main purpose of this section is to establish the necessary estimates for the modified Biot-Savart kernel $H(x,y), K(x,y)$ and $G(x,y)$ given in \eqref{eq def of H}, \eqref{eq sec2 B1} and \eqref{eq def of G(x,y)}, respectively. The estimates for the two-dimensional kernel $H(x,y)$ and $K(x,y)$ are given in Subsection $3.1$ and $3.2$. The estimates for the three-dimensional kernel $G(x,y)$ are given in Subsection $3.3$.
\subsection{Estimates for the modified Biot-Savart kernel $H$ and $K$.}

First, we derive some estimates for the modified Biot-Savart kernel $K(x,y)$ which is given in \eqref{eq sec2 B1}.
\begin{proposition}[Estimates for $K(x,y)$]\label{prop property of K(x,y)} Let $x$, $y $ be two distinct points in $\R^2$, then there holds
\begin{equation}\label{eq K decay x}
|K(x,y)|\lesssim \min\left\{\la y\ra (1+\frac{1}{|x-y|}), \la x\ra (1+\frac{1}{|x-y|})  \right\}.
\end{equation}
\end{proposition}
\begin{proof}
Note that $|\xi(y_1,y_2,a)|\lesssim \langle y \rangle$, so we have
\begin{equation}\label{eq sec3 B0}
|K(x,y)|\lesssim \int_{\R} \frac{\la y\ra}{|x-R_a(y)|^2+a^2}\,da
\end{equation}
and hence it suffices to show that
\begin{equation}\label{eq K decay y}
\int_{\R} \frac{\la y \ra}{|x-R_a(y)|^2+a^2}\,da \lesssim \min\left\{\la y\ra (1+\frac{1}{|x-y|}), \la x\ra (1+\frac{1}{|x-y|})  \right\}.
\end{equation}
 First, we consider the case when $|y|\ge 2|x|$. In this case
 \[
\frac{|x-y|}{3} \le |R_a(x)-y| \le 3|x-y|,
\]
so we have
\[
\int_{\R} \frac{\la y\ra}{|x-R_a(y)|^2+a^2}\,da \lesssim \la y \ra \int_{\R} \frac{1}{|x-y|^2+a^2} \,da \approx \frac{\la y \ra}{|x-y|} \lesssim 1+\frac{\la x \ra}{|x-y|}.
\]

Next, we consider the case when $|y|\le 2|x|$.
Note that $\la y \ra \lesssim \la x \ra$, so it suffices to show
\[
\int_0^{\infty} \frac{\la y\ra}{|x-R_a(y)|^2+a^2} \,da \lesssim \la y \ra \left(1+\frac{1}{|x-y|}\right),
\]
which is equivalent to
\begin{equation}\label{eq sec2 A3}
\int_0^{\infty} \frac{1}{|x-R_a(y)|^2+a^2} \,da \lesssim 1+\frac{1}{|x-y|}.
\end{equation}
To prove \eqref{eq sec2 A3}, we will assume without loss of generality that $2|x| \ge |y| \ge |x|$. On one hand, it follows from Triangle inequality that
\begin{equation*}
|x-y| \le  \left|x-\frac{y}{|y|}|x|\right|+|y|-|x|.
\end{equation*}
On the other hand, the fact $|y| \ge |x|$ implies
\begin{equation*}
\left|x-\frac{y}{|y|}|x|\right| \le |x-y|
\end{equation*}
and
\begin{equation*}
|y|-|x|\le |x-y|,
\end{equation*}
which yield
\begin{equation*}
\left|x-\frac{y}{|y|}|x|\right|+|y|-|x| \lesssim |x-y|.
\end{equation*}
Gathering the estimates above, we finally obtain
\begin{equation}\label{eq distance funciton}
|x-y| \approx  \left|x-\frac{y}{|y|}|x|\right|+|y|-|x|.
\end{equation}
Next we set $\theta_{x,y}=\angle xoy$ and assume without loss of generality that $\theta_{x,y} \in [0,\pi]$, then
\begin{equation}\label{eq sec3 B1}
\theta_{x,y} \approx \frac{\left|x-\frac{y}{|y|}|x|\right|}{|x|}.
\end{equation}
\emph{Case 1: $\theta_{x,y}\ge \theta_0:=10^{-4}$.} In view of \eqref{eq distance funciton} and \eqref{eq sec3 B1}, there holds
\begin{equation*}
\left| x-R_ay \right| \approx |x-y|
\end{equation*}
when $a \le \frac{\theta_0}{2}$. Thus,
\begin{equation*}\begin{aligned}
\int_0^{\infty} \frac{1}{|x-R_ay|^2+a^2}\,da &\lesssim \int_0^{\frac{\theta_0}{2}}\frac{1}{|x-y|^2+a^2}\,da +\int_{\frac{\theta_0}{2}}^\infty \frac{1}{a^2} \,da \\
&\lesssim \frac{1}{|x-y|}+1.
\end{aligned}\end{equation*}
\emph{Case 2: $\theta_{x,y}\le \theta_0$.} In this case, we have
\begin{align}
\int_0^{\infty} \frac{1}{|x-R_ay|^2+a^2}\,da \lesssim&
 \int_{\frac{\pi}{3}}^{\infty} \frac{1}{|x-R_ay|^2+a^2}\,da\label{eq sec3 B2} \\ &+\int_{2\theta_{x,y}}^{\frac{\pi}{3}} \frac{1}{|x-R_ay|^2+a^2}\,da\label{eq sec3 B3} \\ &+  \int_0^{2\theta_{x,y}} \frac{1}{|x-R_ay|^2+a^2}\,da \label{eq sec3 B4}.
\end{align}
For \eqref{eq sec3 B2}, a direct calculation shows that
\begin{equation*}
\int_{\frac{\pi}{3}}^{\infty} \frac{1}{|x-R_ay|^2+a^2}\,da \le \int_{\frac{\pi}{3}}^{\infty} \frac{1}{a^2}\,da \lesssim 1.
\end{equation*}
For \eqref{eq sec3 B3}, $\frac{\pi}{3} \ge a \ge 2\theta_{x,y}$ implies $|x-R_ay| \ge |x-y|$ and hence
\begin{equation*}
\int_{2\theta_{x,y}}^{\frac{\pi}{3}} \frac{1}{|x-R_ay|^2+a^2}\,da \lesssim \int_0^{\infty} \frac{1}{|x-y|^2+a^2} \,da=\frac{1}{|x-y|}.
\end{equation*}

For \eqref{eq sec3 B4}, we have that
\begin{equation*}\begin{aligned}
\int_0^{2\theta_{x,y}} \frac{1}{|x-R_ay|^2+a^2}\,da &\le \int_0^{\theta_{x,y}} \frac{1}{|x-R_ay|^2+a^2}\,da + \int_{\theta_{x,y}}^{2\theta_{x,y}} \frac{1}{|x-R_ay|^2+a^2}\,da \\
  &\le 2\int^{\theta_{x,y}}_{0} \frac{1}{|x-R_ay|^2+a^2}\,da
\end{aligned}\end{equation*}
since $|x-R_ay|=|x-R_{(2\theta_{x,y}-a)}y|$.
To estimate the right hand side, first we use \eqref{eq distance funciton} and \eqref{eq sec3 B1} to conclude that $$|y-x|\approx |x-\frac{y}{|y|}|x||+|y|-|x| \approx |x||\theta_{x,y}|+|y|-|x|,$$ which implies
\begin{equation*}\begin{aligned}
|x-R_ay| &\approx |x||\theta_{x,R_ay}|+|R_ay|-|x| \\  &\approx(\theta_{x,y}-a)|x|+|y|-|x|
\end{aligned}\end{equation*}
for $0 \le a \le \theta_{x,y}$. Thus,
\begin{equation*}\begin{aligned}
\int_0^{2\theta_{x,y}} \frac{1}{|x-R_ay|^2+a^2}\,da  &\lesssim \int_0^{\theta_{x,y}} \frac{1}{(\theta_{x,y}-a)^2|x|^2+(|y|-|x|)^2+a^2}\,da \\
&= \theta_{x,y} \int_0^1 \frac{1}{(1-a)^2|\theta_{x,y}x|^2+(|y|-|x|)^2+|\theta_{x,y}a|^2}\,da.
\end{aligned}\end{equation*}
\emph{Case 2.1, $|y|-|x|\le \theta_{x,y}|x|$.} From \eqref{eq distance funciton} and \eqref{eq sec3 B1}, we see that $|y-x| \approx \theta_{x,y}|x|$ and hence
\begin{equation*}\begin{aligned}
\int_0^{2\theta_{x,y}} \frac{1}{|x-R_ay|^2+a^2}\,da &\lesssim \theta_{x,y}\int_0^1\frac{1}{(1-a)^2\theta_{x,y}^2|x|^2+\theta_{x,y}^2a^2}\,da \\
&= \frac{1}{\theta_{x,y}}\int_0^1 \frac{1}{(1-a)^2|x|^2+a^2}\,da.
\end{aligned}\end{equation*}
Then a direct calculation gives
\begin{equation*}\begin{aligned}
\int_0^1 \frac{1}{(1-a)^2|x|^2+a^2}\,da=&\int_0^1 \frac{1}{a^2|x|^2+(a-1)^2}\,da \\
\le& \int_{-\infty}^{\infty} \frac{1}{a^2|x|^2+(a-1)^2}\,da \\
=&  \frac{\pi}{|x|},
\end{aligned}\end{equation*}
which implies (recalling that $|y-x| \approx \theta_{x,y}|x|$ )
\begin{equation*}\begin{aligned}
\int_0^{2\theta_{x,y}} \frac{1}{|x-R_ay|^2+a^2}\,da &\lesssim \frac{1}{\theta_{x,y}|x|} \lesssim \frac{1}{|x-y|}.
\end{aligned}\end{equation*}

\emph{Case 2.2, $|y|-|x|\ge \theta_{x,y}|x|$.} Using the estimates \eqref{eq distance funciton} and \eqref{eq sec3 B1}, it follows that $|y-x| \approx |y|-|x|$. Therefore,
\begin{equation*}\begin{aligned}
\int_0^{2\theta_{x,y}} \frac{1}{|x-R_ay|^2+a^2}\,da &\lesssim \theta_{x,y} \int_0^1 \frac{1}{(|y|-|x|)^2+|\theta_{x,y}a|^2}\,da \\
&\lesssim \int_0^{\infty} \frac{1}{(|y|-|x|)^2+a^2}\,da \\
&\lesssim \frac{1}{|y|-|x|} \lesssim \frac{1}{|y-x|}.
\end{aligned}\end{equation*}
Gathering these estimates, we get \eqref{eq K decay y} and \eqref{eq K decay x}.
\end{proof}
Next, we derive some estimates for the modified Biot-Savart kernel $H(x,y)$.
\begin{proposition}[Estimates for $H(x,y)$]\label{co mei 1111} Let $x, y$ be two distinct points in $\R^2$ , then there holds
\begin{equation}\label{eq H_1}
|H_1(x,y)|\lesssim \min\left\{\la y\ra (1+\frac{1}{|x-y|}), \la x\ra (1+\frac{1}{|x-y|})  \right\}
\end{equation}
and
\begin{equation}\label{eq H_2}
|H_2(x,y)| \lesssim \min{  \left\{\la x \ra^2(1+\frac{1}{|x-y|}), \la y \ra^2(1+\frac{1}{|x-y|}), \la x\ra \la y\ra(1+\frac{1}{|x-y|})   \right\}            }.
\end{equation}
\end{proposition}

\begin{proof}
We only prove \eqref{eq H_2} since \eqref{eq H_1} is a direct consequence of \eqref{eq sec2 B1}, \eqref{eq def of H1} and \eqref{eq K decay y}. From \eqref{eq def of H2} and \eqref{eq K decay x}, we have that
\begin{equation*}
|H_2(x,y)| \lesssim \la x \ra^2(1+\frac{1}{|x-y|})
\end{equation*}
and
\begin{equation*}
|H_2(x,y)| \lesssim \la y\ra \la x \ra(1+\frac{1}{|x-y|}).
\end{equation*}
Therefore, it remains to show
\begin{equation*}
|H_2(x,y)| \lesssim \la y\ra^2(1+\frac{1}{|x-y|}).
\end{equation*}
In fact, it follows from \eqref{eq def of H2} that
\begin{equation*}\begin{aligned}
|H_2(x,y)|\le \la x\ra|K(x,y)| &\le  \la y\ra\int_{\R} \frac{\la x\ra}{|y-R_ax|^2+a^2}\,da.
\end{aligned}\end{equation*}
Observe that $|y-R_ax|=|x-R_{-a}y|$, so inequality \eqref{eq K decay x} yields
\begin{equation*}
\int_{\R} \frac{\la x\ra}{|y-R_ax|^2+a^2}\,da \lesssim \la y\ra(1+\frac{1}{|x-y|})
\end{equation*}
and hence \eqref{eq H_2} follows.
\end{proof}

Next, we derive some useful estimates for the velocity field $Hw=H_1w+H_2w$.
\begin{proposition}[Estimates for $Hw(x)$]\label{prop boundness of H(x,y)} For any $x\in \R^2$, it holds that
\begin{equation}\label{eq boundness for H_1w}
|H_1w(x)| \lesssim \min \left\{ \|w\|_{L^1_1 \bigcap L^{\infty}_1} ,  \la x\ra\|w\|_{L^1 \bigcap L^{\infty}}      \right\}
\end{equation}
and
\begin{equation*}\label{eq boundness for H_2w}
|H_2w(x)| \lesssim \min \left\{ \la x\ra\|w\|_{L^1_1 \bigcap L^{\infty}_1} , \la x\ra^2\|w\|_{L^1 \bigcap L^{\infty}} , \|w\|_{L^1_2 \bigcap L^{\infty}_2}    \right\}.
\end{equation*}
\end{proposition}
\begin{proof}We only prove \eqref{eq boundness for H_1w} since the other one is similar.
For $H_1w$, recall that
\begin{equation*}\begin{aligned}
|H_1w(x)|=\left| \int_{\R^2} H_1(x,y)w(y) \,dy\right|.
\end{aligned}\end{equation*}
So by \eqref{eq H_1}, on one hand,
\begin{equation}\label{eq sec3 B5}\begin{aligned}
|H_1w(x)| &\lesssim \la x\ra \int_{\R^2} \left(1+\frac{1}{|x-y|}\right)|w(y)|\,dy\\ &\le \la x\ra\left(\|w\|_{L^1}+\int_{\R^2}\frac{|w(y)|}{|x-y|}\,dy\right).
\end{aligned}\end{equation}
On the other hand,
\begin{equation}\begin{aligned}\label{eq sec3 A1}
|H_1w(x)| &\lesssim \int_{\R^2} \left(1+\frac{1}{|x-y|}\right)\la y\ra|w(y)|\,dy\\ &\le \|w\|_{L^1_1}+\int_{\R^2}\frac{|\la y\ra w(y)|}{|x-y|}\,dy.
\end{aligned}\end{equation}
Observe that for a scalar function $f \in L^1 \bigcap L^{\infty}(\R^2)$, there holds
\begin{equation}\begin{aligned}\label{eq sec3 A2}
\int_{\R^2} \frac{|f(y)|}{|x-y|}\,dy \lesssim  \|f\|_{L^1\bigcap L^{\infty}}.
\end{aligned}\end{equation}
So using \eqref{eq sec3 B5}, \eqref{eq sec3 A1} and \eqref{eq sec3 A2}, we  finally get the desired estimates for $H_1w$.
\end{proof}

\subsection{Osgood property of $Hw$.}
In this subsection, we will show that the particle trajectory map $X(\alpha,t)$ given by
\begin{equation*}\begin{cases}\begin{aligned}
\frac{dX(\alpha,t)}{dt}&= Hw(X(\alpha,t),t)\\
X(\alpha,0)&=\alpha
\end{aligned}\end{cases}\end{equation*}
is well-defined. To this end, it suffices to show that the velocity $Hw$ satisfies the Osgood's condition. More precisely, we will show that
\begin{equation*}
\sup_{x,y \in B_R(0)} \frac{|Hw(x)-Hw(z)|}{F(|x-z|)}\lesssim_R 1
\end{equation*}
for $F$ the Log-Lipschitz function given in \eqref{eq def of F}.
Recalling that $Hw(x)=(U^1,U^2)+(-x_2,x_1)U^3$, so it follow that
\begin{equation}\begin{aligned}
\left| Hw(x)-Hw(z) \right| &\lesssim |x-z||U(x)|+\la z\ra|U(x)-U(z)|.
\end{aligned}\end{equation}

Then by the definition of $U$,
\begin{equation*}
U(x)-U(z)=\int_{\R^2} \left(K(x,y)-K(z,y)\right)w(y)\,dy,
\end{equation*}
which implies that (using \eqref{eq sec2 B1})

\begin{equation}\begin{aligned}\label{eq hw(x)-hw(z) 1}
&\left| Hw(x)-Hw(z) \right|  \\ \lesssim& \int_{\R^2} \int_{\R} \left| \frac{(x_1,x_2,0)-\left(R_{a}(y),a\right)}{|(x_1,x_2,0)-\left(R_{a}(y),a\right)|^3}-\frac{(z_1,z_2,0)-\left(R_{a}(y),a\right)}{|(z_1,z_2,0)-\left(R_{a}(y),a\right)|^3}\right|\,da \la z\ra \la y\ra|w(y)|\,dy \\ &+|x-z||U(x)|.
\end{aligned}\end{equation}
For the last term, we use Proposition \ref{prop boundness of H(x,y)} to conclude that
\begin{equation}\label{eq hw(x)-hw(z) 2}
|x-z||U(x)| \lesssim F(x-z)\la x\ra\|w\|_{L^1\bigcap L^{\infty}}
\end{equation}
and
\begin{equation}\label{eq hw(x)-hw(z) 3}
|x-z||U(x)| \lesssim F(x-z)\|w\|_{L^1_1\bigcap L^{\infty}_1}.
\end{equation}
The first term can be bounded by
\begin{equation}\label{eq sec3 B6}
\int_{\R^3}  \left| \frac{(x',0)-(y',y_3)}{|(x',0)-(y',y_3)|^3}-\frac{(z',0)-(y',y_3)}{|(z',0)-(y',y_3)|^3}\right| \la z'\ra\la y' \ra|\Omega^z(y)|\,dy
\end{equation}
after a change of variables $R_ay=\tilde{y}$.
To estimate this integral, we have the following:

\begin{lemma}\label{le sec3 A0}For any $x, z \in \R^3$, it holds that
\begin{equation*}
\int_{\R^3}  \left| \frac{x-y}{|x-y|^3}-\frac{z-y}{|z-y|^3}\right| \la y'\ra|\Omega^z(y)|\,dy \lesssim (\la x\ra+\la z\ra)\|\Omega^z\|_{L^1\bigcap L^{\infty}(\R^2\times \T)}.
\end{equation*}
\end{lemma}
Together with \eqref{eq hw(x)-hw(z) 1} and \eqref{eq sec3 B6}, we see that $Hw$ is locally Osgood continuous and hence the particle trajectory map $X(\alpha,t)$ associated with $Hw$ is well-defined.
The proof of this lemma will need some estimates of the three-dimensional Biot-Savart kernel, which will be postponed to the next subsection.

\subsection{Estimates for the three-dimensional modified Biot-Savart kernel.}
In $\R^3$, the velocity field with helical symmetry can be recovered by the third component of the vorticity:
\begin{equation*}\begin{aligned}
U(x)&=\int_{\R^3} G(x,y) \Omega^z(y)\,dy,
\end{aligned}\end{equation*}
where
\begin{equation}\label{eq def of G(x,y)}
G(x,y):=\frac{x-y}{|x-y|^3} \wedge \xi(y).
\end{equation}
The main purpose of this subsection is to give detailed estimates of $G(x,y)$. First, as a consequence of Lemma \ref{le bound for u in R^3}, we obtain
\begin{equation}\label{eq bound for U in R^3}
|U(x)|\lesssim \|\Omega^z\|_{L^1_1\bigcap L^{\infty}_1(\R^2 \times \mathbb{T})}
\end{equation}
for any $x \in \R^3$. Furthermore, if we assume in addition that $\Omega^z$ is helical, then for any $x\in \R^3$, there holds
\begin{equation*}\begin{aligned}
|U(x)|&\lesssim \la x'\ra\|\Omega^z\|_{L^1\bigcap L^{\infty}(\R^2 \times \mathbb{T})}.
\end{aligned}\end{equation*}
Indeed, by definition of $U$ we have (recall that $w(x_1,x_2)=\Omega^z(x_1,x_2,0)$)
\begin{equation*}\begin{aligned}\label{eq bound for U in R^3 2}
|U(x_1,x_2,0)| &\le \left|\int_{\R^3} \frac{\la y' \ra}{|(x',0)-(y',y_3)|^2} |\Omega^z(y)|\,dy\right|\\
&=\int_{\R^2}|w(y_1,y_2)|\int_{\R}\frac{\la y'\ra}{|x'-R_{y_3}y'|^2+y_3^2}\,dy_3dy'.
\end{aligned}\end{equation*}
Then \eqref{eq K decay y} and Lemma \ref{le norm of R^2 and R^3} yield
\begin{equation*}\begin{aligned}
|U(x_1,x_2,0)|&\lesssim \la x' \ra\|w\|_{L^1\bigcap L^{\infty}(\R^2)} \\
&\lesssim \la x' \ra\|\Omega^z\|_{L^1\bigcap L^{\infty}(\R^2 \times \mathbb{T})},
\end{aligned}\end{equation*}

which implies
\begin{equation*}\begin{aligned}
|U(x',x_3)|&=|U(R_{-x_3}x',0)| \\
&\lesssim \la x' \ra\|\Omega^z\|_{L^1\bigcap L^{\infty}(\R^2 \times \mathbb{T})}
\end{aligned}\end{equation*}
since $U$ is a helical vector field.

Next, we estimate the difference
\begin{equation*}\begin{aligned}
|U(x)-U(z)| &=\left|  \int_{\R^3} \left( \frac{x-y}{|x-y|^3}-\frac{z-y}{|z-y|^3}  \right) \wedge \xi(y)\Omega^z(y)\,dy \right|  \\
&\lesssim \int_{\R^3} \left| \frac{x-y}{|x-y|^3}-\frac{z-y}{|z-y|^3}  \right|\la y' \ra|\Omega^z(y)|\,dy.
\end{aligned}\end{equation*}

\begin{lemma}\label{le K 3d}
Let $\Omega(x)$ be a helical function, then for any $x, z \in \R^3$, it holds that
\begin{equation}\label{eq bound for R3 kernel}
\int_{\R^3} \left| \frac{x-y}{|x-y|^3}-\frac{z-y}{|z-y|^3}  \right| |\Omega(y)|\,dy \lesssim \|\Omega\|_{L^1\bigcap L^{\infty}(\R^2 \times \mathbb{T})}F(|x-z|),
\end{equation}
where $F$ is the Log-Lipschitz function given in \eqref{eq def of F}.
\end{lemma}
\begin{proof}
After changing the variable $ \tilde{y}=y-z$, we obtain
\begin{equation*}
\int_{\R^3} \left| \frac{x-y}{|x-y|^3}-\frac{z-y}{|z-y|^3}  \right| |\Omega(y)|\,dy = \int_{\R^3} \left| \frac{x-z-\tilde{y}}{|x-z-\tilde{y}|^3}-\frac{-\tilde{y}}{|-\tilde{y}|^3}  \right| |\Omega(\tilde{y}+z)|\,d\tilde{y}.
\end{equation*}
Observe that for any $z\in \R^3$,
\begin{equation*}
 \|\Omega(\cdot+z)\|_{L^1\bigcap L^{\infty}(\R^2 \times \mathbb{T})}= \|\Omega\|_{L^1\bigcap L^{\infty}(\R^2 \times \mathbb{T})}.
\end{equation*}
So we may assume without loss of generality that $z=0$ and $\Omega \ge 0$. To complete the proof, it suffices to show that
\begin{equation*}
\int \left| K(x-y)-K(-y) \right|\Omega(y)\,dy \lesssim\|\Omega\|_{L^1\bigcap L^{\infty}(\R^2 \times \mathbb{T})}F(|x|),
\end{equation*}
where $ K(x):=\frac{x}{|x|^3} \in C^{\infty}(\R^3 \setminus 0)$. To this end, we divide the above integral into three parts:
\begin{align}
\int \left| K(x-y)-K(-y) \right|\Omega(y)\,dy =& \int_{|y-x|\ge2}\left| K(x-y)-K(-y) \right|\Omega(y)\,dy \label{eq sec3 C1}\\&+\int_{2 \ge |y-x|\ge 2|x|}\left| K(x-y)-K(-y) \right|\Omega(y)\,dy\label{eq sec3 C2}\\&+\int_{|y-x|\le 2|x|} \left| K(x-y)-K(-y) \right|\Omega(y)\,dy \label{eq sec3 C3}.
\end{align}
For \eqref{eq sec3 C1}, a direct calculation shows that
\begin{equation*}\begin{aligned}
\left| \frac{a}{|a|^3}-\frac{b}{|b|^3} \right|^2 &= \frac{1}{|a|^4}+\frac{1}{|b|^4}-\frac{2a\cdot b}{|a|^3|b|^3} \\
&=\frac{|a|^4+|b|^4-2|a|a\cdot |b|b}{|a|^4|b|^4} \\
&=\frac{\Big| |a|a-|b|b\Big|^2}{|a|^4|b|^4},
\end{aligned}\end{equation*}
which implies
\begin{equation}\begin{aligned}\label{eq a/|a|^3-b/|b|^3}
\left| \frac{a}{|a|^3}-\frac{b}{|b|^3} \right|&=\frac{\Big| |a|a-|b|b\Big|}{|a|^2|b|^2} \\
&\lesssim \frac{\Big| a|a|-a|b|+a|b|-b|b| \Big|}{|a|^2|b|^2} \\
&\lesssim \frac{\Big| |a|-|b| \Big|}{|a||b|^2}+\frac{|a-b|}{|a|^2|b|} \\
&\lesssim \left( \frac{1}{|a||b|^2}+\frac{1}{|a|^2|b|} \right)|a-b|,
\end{aligned}\end{equation}
and hence
\begin{equation*}
|K(x-y)-K(-y)|\lesssim |x|\left( \frac{1}{|x-y|^2|y|}+\frac{1}{|x-y||y|^2} \right).
\end{equation*}
Therefore,
\begin{equation}\label{eq sec3 C9}\begin{aligned}
&\int_{|y-x|\ge2}\left| K(x-y)-K(-y) \right|\Omega(y)\,dy\\ \lesssim& |x| \int_{|x-y|\ge2}
\left( \frac{1}{|x-y|^2|y|}+\frac{1}{|x-y||y|^2} \right)\Omega(y)\,dy.
\end{aligned}\end{equation}
To bound the right-hand side of \eqref{eq sec3 C1}, it remains to show that
\begin{equation}\label{eq sec3 C4}
\int_{|y-x|\ge2} \frac{1}{|x-y||y|^2}\Omega(y)\,dy+\int_{|y-x|\ge2}  \frac{1}{|x-y|^2|y|}\Omega(y)\,dy \lesssim \|\Omega\|_{L^1\bigcap L^{\infty}(\R^2\times \T)}.
\end{equation}
For the first integral above, \eqref{eq le bound for u} gives
\begin{equation}\label{eq sec3 C7}\begin{aligned}
\int_{|y-x|\ge2} \frac{1}{|x-y||y|^2}\Omega(y)\,dy\lesssim \int_{\R^3} \frac{1}{|y|^2} \Omega(y)\,dy \lesssim \|\Omega\|_{L^1\bigcap L^{\infty}(\R^2 \times \mathbb{T})}.
\end{aligned}\end{equation}

For the second integral in \eqref{eq sec3 C4}, let $N=N(x)$ be the integer such that $x \in [2\pi N, 2\pi(N+1))$. Then we get
\begin{align}
\int_{|y-x|\ge2}  \frac{1}{|x-y|^2|y|}\Omega(y)\,dy =&\sum_{-\infty}^{\infty} \int_{\R^2} \int_{2\pi n}^{2\pi(n+1)} \frac{\Omega(y)\chi_{\{ |y-x|\ge 2\}}(y)}{|x-y|^2|y|} \,dy_3 \,dy' \nonumber \\
=& \sum_{|n-N|\ge2}  \int_{\R^2} \int_{2\pi n}^{2\pi(n+1)} \frac{\Omega(y)\chi_{\{ |y-x|\ge 2\}}(y)}{|x-y|^2|y|} \,dy_3 \,dy'\label{eq sec3 C5} \\
&+\sum_{|n-N|\le1} \int_{\R^2} \int_{2\pi n}^{2\pi(n+1)} \frac{\Omega(y)\chi_{\{ |y-x|\ge 2\}}(y)}{|x-y|^2|y|} \,dy_3 \,dy'\label{eq sec3 C6}.
\end{align}
For \eqref{eq sec3 C5}, we only consider the case when $n \ge N+2$ since the other part is similar. Note that $|x-y|\ge |x_3-y_3|\approx |N+1-n|$ for $y_3 \in [2\pi n, 2\pi(n+1))$, so it follows that
\begin{equation*}\begin{aligned}
&\sum_{n=N+2}^{\infty}  \int_{\R^2} \int_{2\pi n}^{2\pi(n+1)} \frac{\Omega(y)\chi_{\{ |y-x|\ge 2\}}(y)}{|x-y|^2|y|} \,dy_3dy_1dy_2 \\ \lesssim& \sum_{n=N+2}^{\infty} \int_{\R^2} \int_{2\pi n}^{2\pi(n+1)} \frac{\Omega(y)}{|N+1-n|^2|y|} \,dy_3\,dy_1dy_2 \\
=&\sum_{n=N+2}^{\infty} \frac{1}{|N+1-n|^2}\int_{\R^2} \int_{2\pi n}^{2\pi(n+1)} \frac{\Omega(y)}{|y|} \,dy_3\,dy_1dy_2.
\end{aligned}\end{equation*}
Denote
\begin{equation*}
\tilde{\Omega}_n(y):=\Omega(y) \chi_{2\pi n \le y_3 \le 2\pi(n+1)},
\end{equation*}

then \eqref{eq standard R3} yields
\begin{equation}\label{eq sec3 C8}\begin{aligned}
&\sum_{n=N+2}^{\infty}  \int_{\R^2} \int_{2\pi n}^{2\pi(n+1)} \frac{\Omega(y)\chi_{\{ |y-x|\ge 2\}}(y)}{|x-y|^2|y|} \,dy_3dy_1dy_2\\ \lesssim& \sum_{n=N+2}^{\infty} \frac{1}{|N+1-n|^2}\|\tilde{\Omega}_n\|_{L^1\bigcap L^{\infty}(\R^3)} \\
\lesssim& \sum_{n=N+2}^{\infty} \frac{1}{|N+1-n|^2}\|\Omega\|_{L^1\bigcap L^{\infty}(\R^2 \times \mathbb{T})} \\
\lesssim& \|\Omega\|_{L^1\bigcap L^{\infty}(\R^2 \times \mathbb{T})}.
\end{aligned}\end{equation}
For \eqref{eq sec3 C6}, denote
\begin{equation*}
\tilde{\Omega}(y)=\Omega(y)\chi_{\left\{2\pi(N-1)\le y_3 \le 2\pi(N+2)\right\}}
\end{equation*}

then in view of \eqref{eq standard R3}, we have
\begin{equation*}\begin{aligned}
&\sum_{|n-N|\le1} \int_{\R^2} \int_{2\pi n}^{2\pi(n+1)} \frac{\Omega(y)\chi_{\{ |y-x|\ge 2\}}(y)}{|x-y|^2|y|} \,dy_3 \,dy_1dy_2\\ \lesssim& \int_{\R^3} \frac{\tilde{\Omega}(y)}{|x-y|^2|y|}\chi_{\{ |y-x|\ge 2\}}\,dy
=\int_{|y-x|\ge 2} \frac{\tilde{\Omega}(y)}{|x-y|^2|y|} \,dy \\
\lesssim& \int_{\R^3} \frac{\tilde{\Omega}(y)}{|y|}\,dy
\lesssim  \|\tilde{\Omega}\|_{L^1\bigcap L^{\infty}(\R^3)}
\lesssim \|\Omega\|_{L^1\bigcap L^{\infty}(\R^2 \times \mathbb{T})}.
\end{aligned}\end{equation*}
Together with \eqref{eq sec3 C7} and \eqref{eq sec3 C8}, we finally obtain \eqref{eq sec3 C4} and hence by \eqref{eq sec3 C9} we get the desired estimates for \eqref{eq sec3 C1}.
For \eqref{eq sec3 C2}, we only consider the case $|x| \le 1$ since otherwise the integral vanishes identically. Note that when $|y| \ge 2|x|$, the segment between $x-y$ and $-y$ does not contain the origin. So the mean value theorem yields
\begin{equation*}\begin{aligned}
\left| K(x-y)-K(-y) \right| &= \left| x \nabla K(-y+\theta x) \right| \\
&\lesssim |x|\frac{1}{|-y+\theta x|^3}
\lesssim \frac{|x|}{|y|^3}
\lesssim \frac{|x|}{|y-x|^3}.
\end{aligned}
\end{equation*}
Therefore,
\begin{equation*}\begin{aligned}
&\int_{2 \ge |y-x|\ge 2|x|}\left| K(x-y)-K(-y) \right|\Omega(y)\,dy \\ \lesssim& |x| \int_{2\ge |y|\ge 2|x|} \frac{\Omega(y)}{|y-x|^3}\,dy
\lesssim \|\Omega\|_{L^{\infty}(\R^2\times \T)} F(|x|).
\end{aligned}\end{equation*}
For \eqref{eq sec3 C3}, we estimate
\begin{equation*}\begin{aligned}
\int_{|y|\le 2|x|} \left| \frac{(x-y)}{|x-y|^3}-\frac{(-y)}{|-y|^3} \right|\Omega(y)\,dy \lesssim \int_{|y|\le 2|x|} \frac{\Omega(y)}{|x-y|^2} \,dy +\int_{|y|\le 2|x|} \frac{\Omega(y)}{|y|^2} \,dy.
\end{aligned}\end{equation*}
A direct calculation shows that
\begin{equation*}\begin{aligned}
\int_{|y|\le 2|x|} \frac{\Omega(y)}{|x-y|^2} \,dy &\lesssim \|\Omega\|_{L^{\infty}} \int_{|y|\le 2|x|}\frac{1}{|x-y|^2} \,dy \\
&\lesssim \|\Omega\|_{L^{\infty}} \int_{|y|\le 2|x|}\frac{1}{|y|^2} \,dy \\
&\lesssim\|\Omega\|_{L^{\infty}} |x|
\end{aligned}\end{equation*}
and similarly,
\begin{equation*}\begin{aligned}
\int_{|y|\le 2|x|} \frac{\Omega(y)}{|y|^2} \,dy&\lesssim \|\Omega\|_{L^{\infty}} \int_{|y|\le 2|x|}\frac{1}{|y|^2} \,dy \\
& \lesssim \|\Omega\|_{L^{\infty}} \int_{|y|\le 2|x|}\frac{1}{|y|^2} \,dy \\
& \lesssim \|\Omega\|_{L^{\infty}} |x|.
\end{aligned}\end{equation*}
Gathering the estimates above we obtain
\begin{equation*}
\int_{|y-x|\le 2|x|} \left| K(x-y)-K(-y) \right|\Omega(y)\,dy \label{eq sec3 C3} \lesssim |x|+F(|x|)+|x| \lesssim F(|x|),
\end{equation*}
which completes the proof.

\end{proof}
\begin{corollary}\label{co log-lip in R^3}For any $x, z\in \R^3$ and $G(x,y)$ defined in \eqref{eq def of G(x,y)}, the following holds.
\begin{equation}\label{eq sec3 A4}
\int_{\R^3} \left|  G(x,y)-G(z,y) \right| |\Omega(y)|\,dy \lesssim  \|\Omega\|_{L^1_1 \bigcap L^{\infty}_1(\R^2\times \T)}F(|x-z|),
\end{equation}
\begin{equation}\label{eq log-lip in R^3}
\int_{\R^3} \left|  G(x,y)-G(z,y) \right| |\Omega(y)|\,dy \lesssim  (\la x'\ra +\la z' \ra)\|\Omega\|_{L^1 \bigcap L^{\infty}(\R^2\times \T)}F(|x-z|)
\end{equation}
and
\begin{equation}\label{eq sec3 A5}
\int_{\R^3} \left|  G(y,x)-G(y,z) \right| |\Omega(y)|\,dy \lesssim (\la x'\ra +\la z' \ra) \|\Omega\|_{L^1 \bigcap L^{\infty}(\R^2\times \T)}F(|x-z|).
 \end{equation}
\end{corollary}
\begin{proof}
The inequality \eqref{eq sec3 A4} follows directly from Lemma \ref{le K 3d}. In order to prove \eqref{eq log-lip in R^3}, observe that
\begin{equation*}\begin{aligned}
\int_{\R^3} \left| G(x,y)-G(z,y) \right| \Omega(y) \,dy &\lesssim \int_{\R^3} \left| K(x-y)-K(z-y) \right| \la y'\ra\Omega(y) \,dy,
\end{aligned}\end{equation*}
so by \eqref{eq bound for R3 kernel} it suffice to consider the integral for $ \max \left\{5|x|,5|z|\right\} \le |y|$. When $|y| \le 1$,
\begin{equation*}\begin{aligned}
\int_{|y| \le 1} \left| G(x,y)-G(z,y) \right| \Omega(y) \,dy &\lesssim \int_{|y| \le 1} \left| K(x-y)-K(z-y) \right| \la y'\ra\Omega(y) \,dy \\
&\lesssim \int_{\R^3} \left| K(x-y)-K(z-y) \right| \Omega(y) \,dy .
\end{aligned}\end{equation*}
Thus \eqref{eq bound for R3 kernel} implies that
\begin{equation*}
\int_{|y| \le 1} \left| G(x,y)-G(z,y) \right| \Omega(y) \,dy \lesssim F(|x-z|) \|\Omega\|_{L^1 \bigcap L^{\infty}(\R^2\times \T)}.
\end{equation*}
When $|y| \ge 1$, we use mean value theorem to conclude that
\begin{equation*}
|K(x-y)-K(z-y)| \lesssim \frac{|x-z|}{|y|^3},
\end{equation*}
which implies (using \eqref{eq le bound for u})
\begin{equation*}\begin{aligned}
&\int_{\max{\{1,5|x|,5|z|\}}\le |y|} \left|  G(x,y)-G(z,y) \right| |\Omega(y)|\,dy \\ \lesssim& |x-z|\int_{\max\{1,5|x|,5|z|\}\le |y|} \frac{\la y'\ra}{|y|^3} |\Omega(y)|\,dy \\
\lesssim& |x-z|\int_{\R^3} \frac{|\Omega(y)|}{|y|^2} \,dy \\
\lesssim& |x-z| \|\Omega\|_{L^1 \bigcap L^{\infty}(\R^2\times \T)}.
\end{aligned}\end{equation*}
Next, we consider \eqref{eq sec3 A5}. Observe that
\begin{equation*}\begin{aligned}
G(y,x)-G(y,z)&=\frac{y-x}{|y-x|^3} \wedge \xi(x)-\frac{y-z}{|y-z|^3}\xi(z)  \\
&=\left(  \frac{y-x}{|y-x|^3}-\frac{y-z}{|y-z|^3} \right)\xi(x)+\frac{y-z}{|y-z|^3}(\xi(x)-\xi(z)),
\end{aligned}\end{equation*}
so we have
\begin{equation*}
\left|  G(y,x)-G(y,z) \right| \lesssim |\xi(x)|\left| \frac{y-x}{|y-x|^3}-\frac{y-z}{|y-z|^3}  \right|+|x-z|\frac{1}{|y-z|^2}.
\end{equation*}
Therefore, \eqref{eq le bound for u} and \eqref{eq bound for R3 kernel} yields
\begin{equation*}\begin{aligned}
&\int \left|  G(y,x)-G(y,z) \right| |\Omega(y)|\,dy\\ \lesssim& \la x'\ra\int \left| \frac{x-y}{|x-y|^3}-\frac{z-y}{|z-y|^3}  \right| |\Omega(y)|\,dy+|x-z|\int \frac{|\Omega(y)|}{|y-z|^2}\,dy \\
\lesssim& \la x'\ra F(x-z) \|\Omega\|_{L^1 \bigcap L^{\infty}(\R^2\times \T)}.
\end{aligned}\end{equation*}
\end{proof}
\section{Uniqueness.}
In this section, we prove the uniqueness of the weak solutions to \eqref{eq 2euler} and \eqref{eq 3euler z} in $L^1_1 \bigcap L^{\infty}_1(\R^2)$ and $L^1_1 \bigcap L^{\infty}_1(\R^3)$, respectively.
\subsection{Uniqueness of the three-dimensional helical Euler equation \eqref{eq 3euler z}.}

Our main theorem is:
\begin{theorem}\label{thm uniqueness in R3}
Assume $\Omega^z,\tilde{\Omega}^z \in L^{\infty}([0,T],L^1_1\bigcap L^{\infty}_1(\R^2\times \T))$ are two lagrangian solutions to \eqref{eq 3euler z} with the same initial data $\Omega^z_0$, then $\Omega^z= \tilde{\Omega}^z$.
\end{theorem}

Before the proof, we first make some useful observations. Recall that $\Omega^z$ satisfies
\begin{equation*}
\partial_t \Omega^z+U \cdot \nabla \Omega^z=0
\end{equation*}
and when $\Omega^z \in L^1_1 \bigcap L^{\infty}_1(\R^2\times \T)$, Corollary \ref{co log-lip in R^3} implies that the velocity field
\begin{equation*}
U(x,t)=\int G(x,y) \Omega^z(y)\,dy
\end{equation*}
with
\begin{equation*}
G(x,y)=\frac{x-y}{|x-y|^3}\wedge \xi(y)=\frac{x-y}{|x-y|^3}\wedge (y_2,-y_1,1)
\end{equation*}
is Log-Lipschitz continuous. Thus, the particle trajectory map
\begin{equation*}\begin{cases}\begin{aligned}
\frac{dX(\alpha,t)}{dt}&= U(X(\alpha,t),t)\\
X(\alpha,0)&=\alpha
\end{aligned}\end{cases}\end{equation*}
is well-defined. Let $X(\alpha,t),\tilde{X}(\alpha,t)$ be the particle trajectory map associated with $U$ and $\tilde{U}$, respectively. Then we have

\begin{lemma}
Assume $X(\alpha,t)=\tilde{X}(\alpha,t)$ for all $\alpha \in supp(\Omega^z_0)$ and $t \in [0,T]$, then $\Omega^z=\tilde{\Omega}^z$.
\end{lemma}
\begin{proof}
We fix $x \in \R^3$ and $t\in [0,T]$, then there exists unique $\alpha$ and $\beta$ such that
$$x=X(\alpha,t)=\tilde{X}(\beta,t)$$
since the map $X(\cdot,t)$ and $\tilde{X}(\cdot,t)$ are one-to-one.

\emph{Case 1: $\alpha,\beta \notin supp \,\Omega^z_0$.} In this case, one has
\begin{equation*}
\Omega^z(x,t)=\Omega^z(X(\alpha,t),t)=\Omega^z_0(\alpha)=0
\end{equation*}
and
\begin{equation*}
\tilde{\Omega}^z(x,t)=\tilde{\Omega}^z(\tilde{X}(\beta,t),t)=\Omega^z_0(\beta)=0,
\end{equation*}
which implies $\Omega^z(x,t)=\tilde{\Omega}^z(x,t)=0$.\\
\emph{Case 2: $\alpha \in supp \,\Omega^z_0$.} Observe that by assumption in our lemma, we have
\begin{equation*}
\tilde{X}(\beta,t)=x=X(\alpha,t)=\tilde{X}(\alpha,t)
\end{equation*}
since $\alpha \in supp \,\Omega^z_0$. Thus, $\alpha=\beta$ since $\tilde{X}(\cdot,t)$ is one-to-one and hence
\begin{equation*}
\Omega^z(x,t)=\Omega^z_0(\alpha,t)=\Omega^z_0(\beta,t)=\tilde{\Omega}^z(x,t).
\end{equation*}
The case when $\beta \in supp \, \Omega^z_0$ is similar so we omit it.
\end{proof}

\begin{remark}From Lemma \ref{le helical particle trajectory map}, we see that $X-\tilde{X}$ is period in $\alpha_3$. So in order to prove $\Omega^z=\tilde{\Omega}^z$, it suffices to show $$X(\alpha,t)=\tilde{X}(\alpha,t) \text{ \quad for all $\alpha \in supp(\Omega^z_0) \bigcap \R^2 \times \T$.}$$
\end{remark}
Now we are ready to prove Theorem \ref{thm uniqueness in R3}.
\begin{proof}[Proof of Theorem \ref{thm uniqueness in R3}]
Motivated by the remark above, we define the distance
\begin{equation*}
D(t):=\int_{\R^2 \times \T} \left|X(\alpha,t)-\tilde{X}(\alpha,t)\right| \la\alpha\ra |\Omega^z_0(\alpha)| \,d\alpha
\end{equation*}
and it remains to show that $D(t)\equiv 0$.
First, we will show that $D(t)$ is continuous and bounded in $[0,T]$. To this end, we estimate
\begin{equation*}\begin{aligned}
|X(\alpha,t)-\alpha|&=\int_0^t \frac{X(\alpha,s)-\alpha}{|X(\alpha,s)-\alpha|} U(X(\alpha,s),s)\,ds \\
&\lesssim T \sup_{t} \|U(\cdot,t)\|_{L^{\infty}},
\end{aligned}\end{equation*}
which together with \eqref{eq bound for U in R^3} gives
\begin{equation*}\begin{aligned}
|X(\alpha,t)-\alpha|&\lesssim \sup_t \|\Omega^z\|_{L^1_1\bigcap L^{\infty}_1(\R^2 \times \mathbb{T})} \lesssim 1.
\end{aligned}\end{equation*}
Similarly,
\begin{equation*}
|\tilde{X}(\alpha,t)-\alpha| \lesssim 1,
\end{equation*}
which implies
\begin{equation}\label{eq sec4 A0}
\la\alpha\ra \approx \la X(\alpha,t)\ra \approx \la\tilde{X}(\alpha,t)\ra
\end{equation}
and
\begin{equation*}
|X(\alpha,t)-\tilde{X}(\alpha,t)|\le |X(\alpha,t)-\alpha|+|\tilde{X}(\alpha,t)-\alpha|\lesssim 1.
\end{equation*}
Therefore,
\begin{equation*}
D(t) \lesssim \int_{\R^2 \times \T} \la\alpha\ra |\Omega^z_0(\alpha)|\,d\alpha=\|\Omega^z\|_{L^1_1(\R^2 \times \mathbb{T})}.
\end{equation*}
Furthermore, the energy $D(t)$ is indeed continuous in $[0,T]$ by Lebesgue dominating convergence Theorem.\\
Next we show $D(t)\equiv 0$. Observe that if $\Omega^z_0 \equiv 0$, then $\Omega^z=\tilde{\Omega}^z \equiv 0$, so we may assume without loss of generality that $\Omega^z_0 \neq 0$. Set $$D^*(t):=\frac{D(t)}{\|\Omega^z_0\|_{L^1_1(\R^2\times \T)}}.$$ We claim that for all $t \in [0,T]$, one has
\begin{equation}\label{eq sec4 B7}
D^*(t)\lesssim \int_0^t F(D^*(s))\,ds.
\end{equation}
Recalling that $D^*(0)=0$, so it follows from Osgood's Lemma that $D^*(t)\equiv 0$ and hence $X(\alpha,t)=\tilde{X}(\alpha,t)$ for all $\alpha \in supp \, (\Omega^z_0) \bigcap \R^2 \times \T$. \\ We now  prove \eqref{eq sec4 B7}. A direct calculation shows that
\begin{align}
\left|  X(\alpha,t)-\tilde{X}(\alpha,t) \right| \lesssim& \int_0^t \left|  U(X(\alpha,s),s)-\tilde{U}(\tilde{X}(\alpha,s),s) \right|\,ds \nonumber \\
\lesssim& \int_0^t \left|  \tilde{U}(X(\alpha,s),s)-\tilde{U}(\tilde{X}(\alpha,s),s) \right|\,ds\\&+\int_0^t \left|  U(X(\alpha,s),s)-\tilde{U}(X(\alpha,s),s) \right|\,ds. \label{eq sec4 B4}
\end{align}
First we use \eqref{eq sec3 A4} to conclude that
\begin{equation*}
|\tilde{U}(x)-\tilde{U}(z)|\lesssim \|\tilde{\Omega}^z\|_{L^1_1\bigcap L^{\infty}_1(\R^2\times \T)} F(|x-z|),
\end{equation*}
which implies
\begin{equation*}
\int_0^t \left|  \tilde{U}(X,s)-\tilde{U}(\tilde{X},s) \right|\,ds \lesssim \int_0^t F\left(|X(\alpha,s)-\tilde{X}(\alpha,s)|\right)\,ds.
\end{equation*}
Since $F$ is a concave function, Jensen's inequality yields
\begin{equation*}\begin{aligned}
&\frac{1}{\|\Omega^z_0\|_{L^1_1}}\int_{\R^2\times \T} \int_0^t \left|  U(X,s)-\tilde{U}(X,s) \right|\,ds \la\alpha\ra |\Omega^z_0(\alpha)|\,d\alpha \\ \lesssim& \int_0^t \int_{\R^2\times \T} F\left(|X(\alpha,s)-\tilde{X}(\alpha,s)|\right) \, \frac{\la\alpha\ra |\Omega^z_0(\alpha)|}{\|\Omega^z_0\|_{L^1_1}}\,d\alpha\,ds \\
\lesssim& \int_0^t F\left(\int_{\R^2 \times \T} |X(\alpha,s)-\tilde{X}(\alpha,s)| \, \frac{\la\alpha\ra |\Omega^z_0(\alpha)|}{\|\Omega^z_0\|_{L^1_1}}\,d\alpha\right)\,ds \\
=&\int_0^t F(D^*(s))\,ds.
\end{aligned}\end{equation*}
Next we estimate the second term in \eqref{eq sec4 B4},
\begin{equation*}\begin{aligned}
|U(x,s)-\tilde{U}(x,s)|&=\left|  \int G(x,y)\Omega^z(y,s)\,dy-\int G(x,y)\tilde{\Omega}^z(y,s)\,dy     \right|\\
&=\left|  \int G(x,X(\beta,s))\Omega^z_0(\beta)\,d\beta -\int G(x,\tilde{X}(\beta,s))\Omega^z_0(\beta)\,d\beta \right| \\
&\lesssim \int \left| G(x,X(\beta,s))-G(x,\tilde{X}(\beta,s)) \right||\Omega^z_0(\beta)|\,d\beta.
\end{aligned}\end{equation*}
Then \eqref{eq sec4 A0} gives
\begin{equation*}\begin{aligned}
 &\int_{\R^2 \times \T} \left|  U(X(\alpha,s),s) -\tilde{U}(X(\alpha,s),s) \right|\la\alpha\ra \Omega^z_0(\alpha)\,d\alpha  \\
 \lesssim& \int_{\R^3} \int_{\R^2\times \T} \left| G(X(\alpha,s),X(\beta,s))-G(X(\alpha,s),\tilde{X}(\beta,s)) \right|\la\alpha\ra|\Omega^z_0(\beta)\Omega^z_0(\alpha)|\,d\alpha \,d\beta \\
 \lesssim& \int_{\R^3} \int_{\R^2\times \T} \left| G(X(\alpha,s),X(\beta,s))-G(X(\alpha,s),\tilde{X}(\beta,s)) \right|\la X(\alpha,s)\ra|\Omega^z_0(\beta)\Omega^z_0(\alpha)|\,d\alpha \,d\beta\\
 :=& A.
\end{aligned}\end{equation*}
Note that $|X(\alpha,t)-\alpha| \lesssim 1$, so there exists $N > 0 $ such that $X(\cdot,t)$ maps $\R^2 \times [0,2\pi]$ into $\R^2 \times [-2\pi N, 2\pi N]$. Thus a change of variables $x=X(\alpha,s)$ yields
\begin{equation*}
A \lesssim \int_{\R^3} \int_{\R^2 \times [-2\pi N,2\pi N]} \left| G(x,X(\beta,s))-G(x,\tilde{X}(\beta,s)) \right|\la x'\ra|\Omega^z(x,s)|\,dx |\Omega^z_0(\beta)|\,d\beta. \\
\end{equation*}
Fix $M \gg N$ such that for all $|\beta_3| \ge 2\pi M$ and all $t \in [0,T]$, $$|X_3(\beta,t)|, |\tilde{X}_3(\beta,t)| \approx |\beta_3|\ge 4\pi N. $$  Then we can divide the integral into two parts according to the value of $\beta_3$. Let $X_M=\R^2 \times [-2\pi M,2\pi M ]$, then
\begin{equation}\label{eq sec4 A5}\begin{aligned}
A \lesssim& \int_{X_M} \int_{X_N} \left| G(x,X(\beta,s))-G(x,\tilde{X}(\beta,s)) \right|\la x'\ra|\Omega^z(x,s)|\,dx |\Omega^z_0(\beta)|\,d\beta \\
&+\int_{X_M^C} \int_{X_N} \left| G(x,X(\beta,s))-G(x,\tilde{X}(\beta,s)) \right|\la x'\ra|\Omega^z(x,s)|\,dx |\Omega^z_0(\beta)|\,d\beta \\
:=&A_1+A_2.
\end{aligned}\end{equation}
For $A_1$, we see directly from \eqref{eq log-lip in R^3} and \eqref{eq sec4 A0} that
\begin{equation*}\begin{aligned}
A_1 &\lesssim  \int_{X_M} \left( \la X(\beta,s)\ra+\la\tilde{X}(\beta,s)\ra \right)F(|X(\beta,s)-\tilde{X}(\beta,s)|)|\Omega^z_0(\beta)|\,d\beta \\
&\lesssim  \int_{X_M} F(|X(\beta,s)-\tilde{X}(\beta,s)|)|\la\beta\ra|\Omega^z_0(\beta)|\,d\beta \\
&\lesssim\int_{\R^2 \times T} F(|X(\beta,s)-\tilde{X}(\beta,s)|)|\la\beta\ra|\Omega^z_0(\beta)|\,d\beta
\end{aligned}\end{equation*}
since $\Omega^z \in L^{\infty}([0,T],L^1_1\bigcap L^{\infty}_1(\R^2\times \T))$. Thus, Jensen's inequality yields
\begin{equation}\label{eq sec4 A6}\begin{aligned}
A_1&\lesssim  F\left(\int_{\R^2 \times \T}|X(\beta,s)-\tilde{X}(\beta,s)|\frac{|\la\beta\ra\Omega^z_0(\beta)|}{\|\Omega^z_0\|_{L^1_1(\R^2\times \T)}}\,d\beta\right).
\end{aligned}\end{equation}
For $A_2$, we estimate
\begin{align}
&\left| G(x,X(\beta,s))-G(x,\tilde{X}(\beta,s)) \right|\nonumber\\ =&\left| \frac{x-X(\beta,s)}{|x-X(\beta,s)|^3}\wedge \xi(X(\beta,s))-\frac{x-\tilde{X}(\beta,s)}{|x-\tilde{X}(\beta,s)|^3}\wedge \xi(\tilde{X}(\beta,s)) \right| \nonumber\\
\lesssim& \left| \left( \frac{x-X(\beta,s)}{|x-X(\beta,s)|^3}-\frac{x-\tilde{X}(\beta,s)}{|x-\tilde{X}(\beta,s)|^3}\right) \wedge \xi(X(\beta,s)) \right| \label{eq sec4 B5} \\
&+\left| \frac{x-\tilde{X}(\beta,s)}{|x-\tilde{X}(\beta,s)|^3}\wedge \left(\xi(X(\beta,s))-\xi(\tilde{X}(\beta,s)) \right) \right|\label{eq sec4 B6}.
\end{align}
For \eqref{eq sec4 B5}, we see from \eqref{eq a/|a|^3-b/|b|^3} that for all $|x| \le 2\pi N$ and $|\beta_3| \ge 2\pi M$,
\begin{equation*}\begin{aligned}
\left| \left( \frac{x-X(\beta,s)}{|x-X(\beta,s)|^3}-\frac{x-\tilde{X}(\beta,s)}{|x-\tilde{X}(\beta,s)|^3}\right) \wedge \xi(X(\beta,s)) \right|
\lesssim& \left|X(\beta,s)-\tilde{X}(\beta,s)\right| \frac{\la \beta' \ra}{|\beta_3|^3} \\
\lesssim& \left|X(\beta,s)-\tilde{X}(\beta,s)\right| \frac{\la \beta' \ra}{|\beta_3|^2}.
\end{aligned}\end{equation*}
For \eqref{eq sec4 B6}, since $\left|x_3-\tilde{X}_3(\beta,s)\right|\approx |\beta_3|$, it is easy to check that
\begin{equation*}
\left| \frac{x-\tilde{X}(\beta,s)}{|x-\tilde{X}(\beta,s)|^3}\wedge \left(\xi(X(\beta,s))-\xi(\tilde{X}(\beta,s)) \right) \right| \lesssim \left|X(\beta,s)-\tilde{X}(\beta,s)\right| \frac{1}{|\beta_3|^2}.
\end{equation*}
Thus,
\begin{equation*}
\left| G(x,X(\beta,s))-G(x,\tilde{X}(\beta,s)) \right| \lesssim \left|X(\beta,s)-\tilde{X}(\beta,s)\right|\frac{\la \beta' \ra}{|\beta_3|^2},
\end{equation*}
which implies that
\begin{equation*}\begin{aligned}
A_2 &\lesssim \|\Omega^z(\cdot,s)\|_{L^1_1(\R^2\times \T)} \int_{X_M^C} \left|X(\beta,s)-\tilde{X}(\beta,s)\right|\frac{\la \beta' \ra}{|\beta_3|^2} |\Omega^z_0(\beta)|\,d\beta \\
&\lesssim \sum_{|n|\ge M} \int_{\R^2} \int_{2\pi n}^{2\pi(n+1)}\left|X(\beta,s)-\tilde{X}(\beta,s)\right|\frac{\la \beta' \ra}{|\beta_3|^2} |\Omega^z_0(\beta)|\,d\beta_3\,d\beta_1d\beta_2 \\
&\lesssim \sum_{|n|\ge M} \int_{\R^2} \int_{2\pi n}^{2\pi(n+1)}\left|X(\beta,s)-\tilde{X}(\beta,s)\right|\frac{\la \beta' \ra}{n^2} |\Omega^z_0(\beta)|\,d\beta_3\,d\beta_1d\beta_2.
\end{aligned}\end{equation*}
Note that $|X(\beta,t)-\tilde{X}(\beta,t)|, \la \beta' \ra$ and $\Omega^z_0(\beta)$ are periodic functions in $\beta_3$, so it follows that
\begin{equation}\label{eq sec4 A7}\begin{aligned}
A_2 &\lesssim \sum_{|n|\ge M} \frac{1}{n^2}\int_{\R^2} \int_{0}^{2\pi}\left|X(\beta,s)-\tilde{X}(\beta,s)\right|\la \beta' \ra |\Omega^z_0(\beta)|\,d\beta_3\,d\beta_1d\beta_2 \\
&\lesssim F \left( \int_{\R^2 \times \T}\left|X(\beta,s)-\tilde{X}(\beta,s)\right| \frac{\la\beta\ra |\Omega^z_0(\beta)|}{\|\Omega^z_0\|_{L^1_1(\R^2\times \T)}} \,d\beta \right),
\end{aligned}\end{equation}
where we have used the fact that $F(r) \gtrsim r$ for $r \ge 0$. Thus, \eqref{eq sec4 A5}, \eqref{eq sec4 A6} and \eqref{eq sec4 A7} gives
\begin{equation*}
A \lesssim F \left( \int_{\R^2 \times \T}\left|X(\beta,s)-\tilde{X}(\beta,s)\right| \frac{\la\beta\ra |\Omega^z_0(\beta)|}{\|\Omega^z_0\|_{L^1_1(\R^2\times \T)}} \,d\beta \right).
\end{equation*}

Integrating the above inequality from $0$ to $t$, we arrive at
\begin{equation*}\begin{aligned}
D^*(t) &\lesssim \int_0^t F\left(\int_{\R^2 \times \T} \left|X(\alpha,s)-\tilde{X}(\alpha,s)\right| \, \frac{\la\alpha\ra |\Omega^z_0(\alpha)|}{\|\Omega^z_0\|_{L^1_1}}\,d\alpha\right)\,ds \\
&=\int_0^t F(D^*(s))\,ds.
\end{aligned}\end{equation*}
This completes the proof.

\end{proof}
\subsection{Uniqueness of the two-dimensional helical Euler equation \eqref{eq 2euler} in $L^1_1 \bigcap L^{\infty}_1(\R^2)$.}We will show that every weak solution to \eqref{eq 2euler} can be lifted to a Lagrangian weak solution of \eqref{eq 3euler z}, thus the uniqueness of \eqref{eq 2euler} follows directly from Theorem \ref{thm uniqueness in R3}.
\begin{lemma}\label{le sec4 lagrangian}
Let $w(x,t)\in L^{\infty}([0,T],L^1 \bigcap L^{\infty}(\R^2))$ be a weak solution to the two-dimensional helical Euler equation \eqref{eq 2euler}, set
\begin{equation}\label{eq sec4 A8}
\Omega(x_1,x_2,x_3,t)=w(R_{-x_3}(x_1,x_2),t)
\end{equation}
and
\begin{equation*}
U(x,t)=\int_{\R^3} G(x,y)\Omega(y,t)\,dy.
\end{equation*}
Then $\Omega(x,t)$ satisfies the three-dimensional helical Euler equation \eqref{eq 3euler z}.
\end{lemma}
\begin{proof}
By definition of the weak solutions, it suffices to show that for any $\phi \in C_c^{\infty}(\R^3\times \R)$,
\begin{equation*}
\int_{\R^3} \Omega(x,t)\phi(x,t) \,dx-\int_{\R^3} \Omega(x,0)\phi(x,0)\,dx=\int_0^t \int_{\R^3} \Omega (\partial_t \phi +U\cdot \nabla \phi) \,dxdt.
\end{equation*}
First we observe that (using \eqref{eq sec4 A8})
\begin{equation*}\begin{aligned}
&\int_{\R^3} \Omega(x,t)\phi(x,t) \,dx-\int_{\R^3} \Omega(x,0)\phi(x,0)\,dx  \\=&\int_{\R^3} w(R_{-x_3}x',t)\phi(x,t) \,dx-\int_{\R^3} w(R_{-x_3}x',0)\phi(x,0)\,dx \\
=&\int_{\R} \left( \int_{\R^2} w(x_1,x_2,t)\phi(R_{x_3}x',x_3,t)\,dx'-\int_{\R^2} w(x_1,x_2,0)\phi(R_{x_3}x',x_3,0)\,dx' \right) \,dx_3,
\end{aligned}\end{equation*}
which implies (we use the notation $\overline{\nabla}=(\partial_1,\partial_2)$ for simplicity)
\begin{equation*}\begin{aligned}
&\int_{\R^3} \Omega(x,t)\phi(x,t) \,dx-\int_{\R^3} \Omega(x,0)\phi(x,0)\,dx \\=&\int_{\R} \left( \int_0^t \int_{\R^2} w(x_1,x_2,s)\left[ \partial_s +Hw(x_1,x_2,s)\cdot \overline{\nabla} \right]\left( \phi(R_{x_3}x',x_3,s) \right) \,dx'ds \right)  \,dx_3  \\
=&\int_0^t\int_{\R} \left(  \int_{\R^2} w(x_1,x_2,s)\left[ \partial_s +Hw(x_1,x_2,s)\cdot \overline{\nabla} \right]\left( \phi(R_{x_3}x',x_3,s) \right) \,dx' \right)  \,dx_3 \,ds
\end{aligned}\end{equation*}
since $w$ is a weak solution to \eqref{eq 2euler}.
Therefore, to complete the proof of the lemma, it remains to show that
\begin{equation}\label{eq sec4 A9}\begin{aligned}
&\int_{\R^3} \Omega(x,s) \left[ \partial_s+U\cdot \nabla \right]\phi(x,s)\,dx \\
=&\int_{\R} \left(  \int_{\R^2} w(x_1,x_2,s)\left( \partial_s +Hw(x_1,x_2,s)\cdot \overline{\nabla} \right)\left( \phi(R_{x_3}x',x_3,s) \right) \,dx' \right)  \,dx_3.
\end{aligned}\end{equation}
Since $\Omega$ is a helical function, we have
\begin{equation}\label{eq sec4 B0}\begin{aligned}
\int_{\R^3} \Omega(x,s) \partial_s \phi(x,s)\,dx &=\int_{\R^3} w(R_{-x_3}x',s) \partial_s \phi(x,s)\,dx \\
&=\int_{\R^3}w(x_1,x_2,s) \partial_s \phi(R_{x_3}x',x_3,s)\,dx.
\end{aligned}\end{equation}
Recall that
\begin{equation}\label{eq mei hw}
Hw(x_1,x_2)=(U^1(x',0),U^2(x',0))+(-x_2,x_1)U^3(x',0),
\end{equation}
where $U^3(\cdot,t)$ are helical functions on $\R^3$.  So we use \eqref{eq sec2 A0}, \eqref{eq partial3 scalar}, \eqref{eq mei hw} and the fact that $\Omega U^3$ is a helical function to get
\begin{equation}\label{eq mei 88888}\begin{aligned}
&\int_{\R^3} \Omega(x,s)(U^1\partial_1\phi+U^2\partial_2\phi) \,dx +\int_{\R^3} (\Omega U^3) \partial_3 \phi \,dx \\=& \int_{\R^3} \Omega(x,s)(U^1\partial_1\phi+U^2\partial_2\phi) \,dx +\int_{\R^3} \Omega (x_1U^3\partial_2\phi -x_2U^3\partial_1 \phi ) \,dx.
\end{aligned}\end{equation}
Note that for a smooth function $\phi =\phi(x_1,x_2)$,
\begin{equation*}
\overline{\nabla} \left( \phi(R_{x_3}x') \right) =R_{-x_3}\left(\nabla \phi(R_{x_3}x')\right).
\end{equation*}
So in view of \eqref{eq mei 88888}, there holds
\begin{equation}\begin{aligned}
&\int_{\R^3} \Omega(x,s)(U^1\partial_1\phi+U^2\partial_2\phi) \,dx +\int_{\R^3} (\Omega U^3) \partial_3 \phi \,dx \\
=& \int_{\R} \int_{\R^2} w(x_1,x_2,s) Hw(x_1,x_2,s)\cdot \overline{\nabla} \left( \phi(R_{x_3}x',x_3,s) \right) \,dx_1dx_2   \,dx_3.
\end{aligned}\end{equation}
Together with \eqref{eq sec4 B0} we finally get \eqref{eq sec4 A9} and hence complete the proof.

\end{proof}
\begin{lemma}\label{le sec4 A0}
Assume $w(x,t)\in L^{\infty}([0,T],L^1 \bigcap L^{\infty}(\R^2))$ is Lagrangian weak solution of \eqref{eq 2euler}, then
\begin{equation}\label{eq sec4 B3}
\Omega(x',x_3,t)=w(R_{-x_3}x',t)
\end{equation}
is a Lagrangian weak solution to \eqref{eq 3euler z}.
\end{lemma}
\begin{proof}
Let $X(\alpha_1,\alpha_2,t)=(X_1(\alpha_1,\alpha_2,t),X_2(\alpha_1,\alpha_2,t))$ be the particle trajectory map associated to $Hw$. Motivated by \eqref{eq helical particle map}, we define
\begin{equation*}
X_3(\alpha_1,\alpha_2,t)=\int_0^t U_3(X_1(\alpha_1,\alpha_2,s),X_2(\alpha_1,\alpha_2,s),0)\,ds,
\end{equation*}
and
\begin{equation}\label{eq sec4 B2}
Y(\alpha_1,\alpha_2,0,t)=(R_{X_3}(X_1,X_2),X_3).
\end{equation}
We claim that
\begin{equation}\label{eq sec4 B1}
Y(\alpha_1,\alpha_2,\alpha_3,t):=S_{\alpha_3}Y(R_{-\alpha_3}(\alpha_1,\alpha_2),0,t)
\end{equation}
is the particle trajectory map associated with $U$ in $\R^3$. That is,
\begin{equation*}\begin{cases}\begin{aligned}
\frac{dY(\alpha,t)}{dt}&= U(Y(\alpha,t),t)\\
Y(\alpha,0)&=\alpha.
\end{aligned}\end{cases}\end{equation*}
Indeed, a direct calculation shows that
\begin{equation*}\begin{aligned}
\frac{dY_1(\alpha_1,\alpha_2,0,t)}{dt}=U_1\cos(X_3)+U_2\sin(X_3),
\end{aligned}\end{equation*}

\begin{equation*}
\frac{dY_2(\alpha_1,\alpha_2,0,t)}{dt}=-U_1\sin(X_3)+U_2\cos(X_3)
\end{equation*}
and
\begin{equation*}
\frac{dY_3(\alpha_1,\alpha_2,0,t)}{dt}=U_3.
\end{equation*}
Hence,
\begin{equation*}
\frac{dY(\alpha_1,\alpha_2,0,t)}{dt}=R_{X_3}U(X_1,X_2,0)=U(R_{X_3}(X_1,X_2),X_3)
\end{equation*}
since $U$ is a helical vector field. Then \eqref{eq sec4 B2} yields
\begin{equation*}
\frac{dY(\alpha_1,\alpha_2,0,t)}{dt}=U(Y(\alpha_1,\alpha_2,0,t),t),
\end{equation*}
which combined with \eqref{eq sec4 B1}, gives
\begin{equation*}\begin{aligned}
\frac{dY(\alpha_1,\alpha_2,\alpha_3,t)}{dt}&=\frac{dS_{\alpha_3}Y(R_{-\alpha_3}(\alpha_1,\alpha_2),0,t)}{dt} \\
&=\frac{dR_{\alpha_3}Y(R_{-\alpha_3}(\alpha_1,\alpha_2),0,t)}{dt} \\
&=R_{\alpha_3}U(Y(R_{-\alpha_3}(\alpha_1,\alpha_2),0,t),t).
\end{aligned}\end{equation*}
Finally, using \eqref{eq helical vector} and \eqref{eq sec4 B1}, we have that
\begin{equation*}\begin{aligned}
\frac{dY(\alpha_1,\alpha_2,\alpha_3,t)}{dt}&=R_{\alpha_3}U(Y(S_{-\alpha_3}\alpha,t),t) \\
&=U(S_{\alpha_3}Y(S_{-\alpha_3}\alpha,t),t) \\
&=U(Y(\alpha_1,\alpha_2,\alpha_3,t),t),
\end{aligned}\end{equation*}
which implies that $Y$ is the particle trajectory map of $U$ in $\R^3$ and hence completes the proof of the claim. \\Now it remains to show that
\begin{equation*}\begin{aligned}
\Omega(Y(\alpha,t),t) =\Omega_0(\alpha).
\end{aligned}\end{equation*}
To this end, we set $\beta=R_{-\alpha_3}(\alpha_1,\alpha_2)$. Then it follows from \eqref{eq helical function} and \eqref{eq helical particle map} that
\begin{equation*}\begin{aligned}
\Omega(Y(\alpha,t),t)&=\Omega (S_{\alpha_3}Y(\beta,0,t),t) \\
&=\Omega (Y(\beta,0,t),t) \\
&=\Omega (S_{-Y_3(\beta,0,t)}Y(\beta,0,t),t),
\end{aligned}\end{equation*}
which yields (using \eqref{eq sec4 B2} and \eqref{eq sec4 A8})
\begin{equation*}\begin{aligned}
\Omega(Y(\alpha,t),t)&=\Omega (S_{-X_3(\beta,t)}(R_{X_3}(X_1,X_2),X_3)(\beta,t),t) \\
&=\Omega (X_1(\beta,t),X_2(\beta,t),0,t) \\
&=w (X_1(\beta,t),X_2(\beta,t),t)\\
&=w_0(\beta)
\end{aligned}\end{equation*}
since $w$ is a Lagrangian solution.
Thus in view of \eqref{eq sec4 B3}, we get
\begin{equation*}\begin{aligned}
\Omega(Y(\alpha,t),t)
 =\Omega_0(\alpha),
\end{aligned}\end{equation*}
which completes the proof.
\end{proof}

 An immediate consequence of Lemma \ref{le sec4 A0} is:
\begin{corollary}
Every weak solution $w(x,t)$ of \eqref{eq 2euler} can be lifted to a weak Lagrangian solution to \eqref{eq 3euler z} with $\Omega^z(x,t)=w(R_{-x_3}(x_1,x_2),t)$. Moreover, such lifting is injective.
\end{corollary}
\begin{proof}
It follows from Lemma \ref{le weak solution is lagrangian} that every weak solution $w(\cdot,t)$ to \eqref{eq 2euler} in $L^1_1\bigcap L^{\infty}_1(\R^2)$ is indeed a Lagrangian weak solution. So Lemma \ref{le sec4 lagrangian} implies that every weak solution to \eqref{eq 2euler} can be lifted to a weak Lagrangian solution to \eqref{eq 3euler z}. Thus it remains to show that such lifting is injective. Let $w$ and $\tilde{w}$ be two different weak solution of \eqref{eq 2euler}, then
\begin{equation*}
\|\Omega^z(\cdot,t)-\tilde{\Omega}^z(\cdot,t)\|_{L^1_1\bigcap L^{\infty}_1(\R^2 \times \T)} \approx \|w(\cdot,t)-\tilde{w}(\cdot,t)\|_{L^1_1 \bigcap L^{\infty}_1(\R^2)} \neq 0
\end{equation*}
as a consequence of Lemma \ref{le norm of R^2 and R^3}.
\end{proof}
Together with Theorem \ref{thm uniqueness in R3}, we obtain the uniqueness of weak solutions to the two-dimensional helical Euler equation.

\begin{corollary}\label{co uniqueness}
Assume $w,\tilde{w} \in L^{\infty}([0,T],L^1_1\bigcap L^{\infty}_1(\R^2))$ are two weak solutions to \eqref{eq 2euler} with the same initial data $w_0$, then $w=\tilde{w}$.
\end{corollary}

\section{Global existence of weak solutions and conservation laws.}

The main purpose of this section is to show that all weak solutions to the two-dimensional helical Euler equation \eqref{eq 2euler} are global in $L^1_1 \cap L^{\infty}_1(\R^2)$. Furthermore, this section also establishes the conservation of energy and the second momentum.

\subsection{A formal argument}
We will assume in this subsection that all functions are smooth enough and exhibit sufficient decay at
infinity. The well-known Beale-Kato-Majda criterion suggests that a solution $\Omega$ of the three-dimensional Euler equation \eqref{eq 3euler} blows-up at time $T$ if and only if
\begin{equation*}
\int_0^T \|\Omega\|_{L^{\infty}(\R^3)}\,dt= \infty.
\end{equation*}
For $\Omega=(x_2\Omega^z,-x_1\Omega^z,\Omega^z)$ and
 $\Omega^z(x_1,x_2,x_3)=w(R_{-x_3}(x_1,x_2))$, we have

\begin{equation*}
\|\Omega\|_{L^{\infty}(\R^3)}=\|\Omega^z\|_{L^{\infty}_1(\R^3)}=\|w\|_{L^{\infty}_1(\R^2)}.
\end{equation*}
Recall that $w$ satisfies the (transport) equation
\begin{equation*}
\partial_t w +Hw\cdot \nabla w=0
\end{equation*}
with $\nabla \cdot Hw=0$. Multiply both sides by $\la x\ra$, we obtain
\begin{equation*}\begin{aligned}
\partial_t(\la x\ra w)+Hw \cdot \nabla (\la x\ra w)&=w Hw \cdot \nabla \la x\ra \\
&=w H_1w \cdot \nabla \la x\ra
\end{aligned}\end{equation*}
since $Hw=H_1w+H_2w$ and $x\cdot H_2w$=0. Therefore,
\begin{equation*}\begin{aligned}
\|\la x\ra w(x,t)\|_{L^1 \bigcap L^{\infty}} \le \|\la x\ra w_0(x)\|_{L^1 \bigcap L^{\infty}}+\int_0^t \|wH_1w(x,s)\|_{L^1 \bigcap L^{\infty}} \,ds .
\end{aligned}\end{equation*}
To bound the second term on the right-hand side, we use \eqref{eq boundness for H_1w} to estimate
\begin{equation*}\begin{aligned}
\|wH_1w\|_{L^1 \bigcap L^{\infty}} &\le \|Hw\|_{L^{\infty}}\|w\|_{L^1 \bigcap L^{\infty}} \\
& \lesssim  \|w\|_{L^1_1 \bigcap L^{\infty}_1}\|w\|_{L^1 \bigcap L^{\infty}},
\end{aligned}\end{equation*}
which implies (recall that $\|w(\cdot,t)\|_{L^p}=\|w_0\|_{L^p}$ since $\nabla \cdot Hw=0$)
\begin{equation*}
\|w(\cdot,t)\|_{L^1_1 \bigcap L^{\infty}_1} \le \|w_0\|_{L^1_1 \bigcap L^{\infty}_1}+C\|w_0\|_{L^1 \bigcap L^{\infty}} \int_0^t \|w(\cdot,s)\|_{L^1_1 \bigcap L^{\infty}_1}\,ds.
\end{equation*}
So it follows from Gronwall's inequality that $\|w(\cdot,t)\|_{L^1_1 \bigcap L^{\infty}_1} \lesssim Ce^{Ct}$ for all $t \ge 0$ and hence the Beale-Kato-Majda criterion implies at least formally that solutions to \eqref{eq 2euler} are global. More precisely, we have
\begin{theorem}\label{thm existence of smooth solution}
For any $w_0 \in C_c^{\infty}(\R^2)$, there exist a global Lagrangian weak solution $w(x,t)$ to \eqref{eq 2euler} with initial vorticity $w_0$. Furthermore, for any $t \in [0,T]$, there exists a constant $C$ depending only on $T$ and $\|w_0\|_{L^1_1\bigcap L^{\infty}_1(\R^2)}$ such that
\begin{equation*}
\sup_{t\in[0,T]}\|w(\cdot,t)\|_{L^1_1 \bigcap L^{\infty}_1(\R^2)} \le C_{T,w_0}.
\end{equation*}
\end{theorem}
\begin{remark}
By the same argument, one can show that for all $T>0$, $N \in \mathbb{N}^*$ and $t \in [0,T]$, there exist some constant $C$ depending only on $N,T$ and $\|w_0\|_{L^1_N \bigcap L^{\infty}_N(\R^2)}$ such that
\begin{equation*}
\sup_{t\in[0,T]}\|w(\cdot,t)\|_{L^1_N \bigcap L^{\infty}_N(\R^2)} \le C_{N,T,w_0}.
\end{equation*}
\end{remark}
The proof of Theorem \ref{thm existence of smooth solution} is standard, which is sketched in Appendix B.

Next we prove Theorem \ref{thm conservation}. By multiplying both side of \eqref{eq 3euler} by $(-\Delta_{\R^2\times \T})^{-1}$ and integrating, we see directly that
$$
E(t):=\int_{\R^2\times \T} (-\Delta_{\R^2\times \T})^{-1}\Omega \cdot \Omega \,dx
$$
is conserved. For the second momentum, we take derivatives on $M_1(t)$ directly.
\begin{equation*}\begin{aligned}
\frac{dM_1(t)}{dt}&=-\int_{\R^2\times \T} |x'|^2 \nabla \cdot (U\Omega^z) \,dx \\
&=2\int_{\R^2\times \T} (x',0)\cdot U(x) \Omega^z(x)\,dx\\
&=2\int_{\R^2\times \T} \int_{\R^2\times \T} \mathfrak{g}(x,y) \Omega^z(y) \Omega^z(x)\,dy\,dx,
\end{aligned}\end{equation*}
where (recall that $\mathbf{G}$ is the Green's function defined in Proposition \ref{prop stream function})
\begin{equation*}\begin{aligned}
\mathfrak{g}(x,y):=&(x',0) \cdot \nabla \mathbf{G}(x-y) \wedge \xi(y)\\
=&\Big(x_1\partial_2\mathbf{G}(x-y)-x_2\partial_1\mathbf{G}(x-y)\Big) +(x_1y_1+x_2y_2)\partial_3\mathbf{G}(x-y).
\end{aligned}\end{equation*}
Therefore, it follows from \eqref{eq sec2 C3} and \eqref{eq sec2 C4} that
$$
\mathfrak{g}(x,y)+\mathfrak{g}(y,x) \equiv 0
$$
and hence
\begin{equation*}\begin{aligned}
\frac{dM_1(t)}{dt}=\int_{\R^2\times \T} \int_{\R^2\times \T} (\mathfrak{g}(x,y)+\mathfrak{g}(x,y)) \Omega^z(y) \Omega^z(x)\,dy\,dx=0.
\end{aligned}\end{equation*}
 Finally, for general solutions $\Omega^z \in L^1_2\bigcap L^{\infty}_2(\R^2\times \T)$, the conservation law can be obtained by using a similar approximation argument as mentioned above, which we omit here for brevity.

\subsection{Proof of the existence.}In this subsection, we will show that for every $w_0 \in L^1\bigcap L^{\infty}(\R^2)$, there exists a global Lagrangian weak solution to \eqref{eq 2euler} with initial data $w_0$.

    \begin{proof}[Proof of Theorem \ref{thm main} (i)]Let $\{w_{0,n}\} \in C_c^{\infty}(\R^2)$  such that $w_{0,n} \to w_0 \in L^1(\R^2)$ and $\|w_{0,n}\|_{L^{\infty}(\R^2)} \lesssim \|w_0\|_{L^{\infty}(\R^2)}$.  Let $w_n(\cdot,t)$ be a sequence of Lagrangian solutions to \eqref{eq 2euler} with initial data $w_{0,n}$. Then the fact $\nabla \cdot Hw_n=0$ yields
\begin{equation*}
\|w_n(\cdot,t)\|_{L^1 \bigcap L^\infty(\R^2)}=\|w_{0,n}\|_{L^1 \bigcap L^{\infty}(\R^2)} \lesssim 1,
\end{equation*}
which implies that (using Proposition \ref{prop boundness of H(x,y)})
\begin{equation*}
|H_1w_n(x,t)| \lesssim \la x\ra \quad \text{and} \quad |H_2w_n(x,t)| \lesssim \la x\ra^2.
\end{equation*}
Then by \eqref{eq hw(x)-hw(z) 1}, \eqref{eq hw(x)-hw(z) 2} and \eqref{eq log-lip in R^3}, we conclude that
\begin{equation*}
|Hw_n(x,t)-Hw_n(z,t)| \lesssim (\la x\ra^3+\la z\ra^3)F(|x-z|).
\end{equation*}
Thus, it follows from Lemma \ref{le app1} that for any $R, T>0$, $X_n(\alpha,t)$ and $X_n^{-t}(\alpha)$ are uniformly bounded and equicontinuous on $B_R \times [0,T]$. So there exist $X(\alpha,t)$ and its inverse map $X^{-t}(\alpha)$ and a subsequence $X_{n_k}$ such that
\begin{equation*}\begin{aligned}
X_{n_k}(\alpha,t) \to X(\alpha,t)
\end{aligned}\end{equation*}
and
\begin{equation*}\begin{aligned}
X_{n_k}^{-t}(\alpha) \to X^{-t}(\alpha)
\end{aligned}\end{equation*}
uniformly in every compact set. Furthermore, for any $t\ge 0$, the map $X^{-t}(\cdot)$ and $X(\cdot,t) $ preserves Lebesgue measure. Now setting $w(x,t)=w_0(X^{-t}(x))$, we will show that $w(x,t)$ is the desired solution.
To this end, first we write the subsequence $n_k $ as $n$ for simplicity and we claim that for every $t \in [0,T]$,
\begin{equation}\label{eq sec6 A0}
\|w(\cdot,t)-w_n(\cdot,t)\|_{L^1} \to 0.
\end{equation}
Indeed, since $w_n(\cdot)$ is a Lagrangian weak solution, we have
\begin{equation*}\begin{aligned}
\|w(\cdot,t)-w_n(\cdot,t)\|_{L^1}
&=\int_{\R^2} |w_0(X^{-t}(x))-w_0(X_n^{-t}(x))\,dx.
\end{aligned}\end{equation*}
Now for any $\e >0$, let $w_c \in C_c^{\infty}(\R^2)$ such that
\begin{equation*}
\|w_0-w_c\|_{L^1} \le \frac{\e}{100}.
\end{equation*}
Then it follows that
\begin{equation*}\begin{aligned}
\|w(\cdot,t)-w_n(\cdot,t)\|_{L^1}&\le \int_{\R^2} |w_0(X^{-t}(x))-w_c(X^{-t}(x))|\,dx+\int_{\R^2} |w_c(X^{-t}(x))-w_c(X_n^{-t}(x))|\,dx\\&+\int_{\R^2} |w_c(X_n^{-t}(x))-w_0(X_n^{-t}(x))|\,dx.
\end{aligned}\end{equation*}
Observe that $X^{-t}(\cdot)$ and $X(\cdot,t) $ preserves Lebesgue measure, so there holds
\begin{equation*}
\int_{\R^2} |w_0(X^{-t}(x))-w_c(X^{-t}(x))|\,dx= \|w_0-w_c\|_{L^1} \le \frac{\e}{100}.
\end{equation*}
Similarly,
\begin{equation*}
\int_{\R^2} |w_0(X_n^{-t}(x))-w_c(X_n^{-t}(x))|\,dx= \|w_0-w_c\|_{L^1} \le \frac{\e}{100}.
\end{equation*}
Then the dominating convergence theorem gives
   $$\int_{\R^2} |w_c(X^{-t}(x))-w_c(X_n^{-t}(x))|\,dx \le \frac{97\e}{100}$$
for $n$ large enough since $w_c \in C_c^{\infty}(\R^2)$ and $X_n^{-t}(x) $ converges uniformly to $X^{-t}(x)$ in compact set. This completes the proof of the claim. Furthermore, it follows from \eqref{eq bound for Hw by only L1 norm} that
\begin{equation}\label{eq sec6 A1}
Hw_n(x,t) \to Hw(x,t)
\end{equation}
for all $x \in \R^2$ and $t \in \R$. Note that for any $\phi \in C_c^{\infty}(\R^2)$, we have
\begin{equation*}
\int w_n(x,t)\phi(x,t) \,dx-\int w_n(x,0)\phi(x,0)\,dx=\int_0^t \int_{\R^2} w_n (\partial_t \phi +Hw_n\cdot \nabla \phi) \,dxdt
\end{equation*}
since each $w_n$ is a weak solution. So with the help of \eqref{eq sec6 A0} and \eqref{eq sec6 A1}, letting $n\to \infty$, the desired conclusion is reached.
\end{proof}
\section{Continuous dependence on initial data and conservation laws. }In this section we will prove the continuous dependence on initial data. First, we prove the following simplified version.
\begin{theorem}
Let $w_{0,n}$ be a sequence of initial data such that
\begin{equation*}
\sup_n \|w_{0,n}\|_{L^{\infty}_1(\R^2)} < \infty
\end{equation*}
and
\begin{equation*}
\|w_{0,n}-w_0\|_{L^1_1(\R^2)} \to 0.
\end{equation*}
Then for any $t \ge 0$, one has
\begin{equation*}
 \|w_n(t)-w(t)\|_{L^1_1(\R^2)} \to 0
\end{equation*}
as $n \to \infty$.
\end{theorem}
\begin{proof}
First we claim that for any $R,T>0$, $X_n(\cdot,\cdot) \to X(\cdot,\cdot)$ uniformly in $B_R \times [0,T]$. Assume this is not true, then there exist $R,T,\delta>0$, $(\alpha_k,t_k) \in B_R \times [0,T]$ and a subsequence $X_{n_k}$ such that
\begin{equation*}
|X_{n_k}(\alpha_k,t_k)-X(\alpha_k,t_k)| \ge \delta>0.
\end{equation*}
Observe that $B_R \times [0,T]$ is compact, so we may assume without loss of generality that $(\alpha_k,t_k) \to (\alpha_0,t_0) \in B_R\times [0,T]$. Then by a similar argument as in Section $6$, there exists $\tilde{X}(\alpha,t)$ such that $X_{n_k} \to \tilde X$ uniformly in compact set. Thus, if we set $w(\tilde{X}(\alpha,t),t):=w_0(\alpha)$, then $\tilde{w}$ is also a weak solution to \eqref{eq 2euler} with initial data $w_0$. Meanwhile, from Corollary \ref{co uniqueness} we know that solutions to \eqref{eq 2euler} is unique in $L^1_1\bigcap L^{\infty}_1(\R^2)$, so $w=\tilde{w}$ and hence $X= \tilde{X}$. Thus,
\begin{equation*}
X(\alpha_0,t_0)-\tilde{X}(\alpha_0,t_0)=0,
\end{equation*}
which leads to a contradiction since
\begin{equation*}
|X(\alpha_0,t_0)-\tilde{X}(\alpha_0,t_0)|=\lim_{k\to \infty}|X(\alpha_k,t_k)-X_{n_k}(\alpha_k,t_k)| \ge \delta>0.
\end{equation*}
Therefore, $X_n \to X$ uniformly in compact set. \\Next we show $w_n \to w$ in $L^1_1(\R^2)$.  A direct calculation shows that
\begin{align}
\|w(\cdot,t)-w_n(\cdot,t)\|_{L^1_1(\R^2)}
=&\int_{\R^2} \la x\ra |w_{0,n}(X_n^{-t}(x))-w_0(X^{-t}(x))|\,dx \nonumber\\
\le& \int_{\R^2} \la x\ra|w_{0,n}(X_n^{-t}(x))-w_0(X_n^{-t}(x))|\,dx \label{eq sec7 A0}\\&+\int_{\R^2} \la x\ra|w_0(X_n^{-t}(x))-w_0(X^{-t}(x))|\,dx. \label{eq sec7 A1}
\end{align}
For \eqref{eq sec7 A0}, we use \eqref{eq sec4 A0} to conclude that for $n$ large enough
\begin{equation*}\begin{aligned}
\int_{\R^2} \la x\ra|w_{0,n}(X_n^{-t}(x))-w_0(X_n^{-t}(x))|\,dx&=\int_{\R^2} \la X_n(\alpha,t)\ra|w_{0,n}(\alpha)-w_0(\alpha)|\,d\alpha \\
&\lesssim \int_{\R^2} \la\alpha\ra |w_{0,n}(\alpha)-w_0(\alpha)|\,d\alpha \\
&\le \frac{\e}{100}
\end{aligned}\end{equation*}
since $w_{0,n}\to w_0$ in $L^1_1(\R^2)$.
For \eqref{eq sec7 A1}, we choose $w_{0,\e} \in C_c^{\infty}(\R^2)$ such that $\|w_{0,\e}-w_0\|_{L^1_1(\R^2)} \le \delta \e$ for some $\delta$ small enough. Then it follows that
\begin{align}
\int_{\R^2} \la x\ra|w_0(X_n^{-t}(x))-w_0(X^{-t}(x))|\,dx&\le \int_{\R^2} \la x\ra|w_0(X_n^{-t}(x))-w_{0,\e}(X_n^{-t}(x))|\,dx \label{eq sec7 A2}\\ &+\int_{\R^2} \la x\ra|w_{0,\e}(X_n^{-t}(x))-w_{0,\e}(X^{-t}(x))|\,dx \label{eq sec7 A3}\\ &+\int_{\R^2} \la x\ra|w_{0,\e}(X^{-t}(x))-w_{0}(X^{-t}(x))|\,dx\label{eq sec7 A4}.
\end{align}
Recall that $X^{-t}(\cdot)$ and $X_n^{-t}(\cdot)$ preserve Lebesgue measure, so in view of \eqref{eq sec7 A0}, the integral in \eqref{eq sec7 A2} and \eqref{eq sec7 A4} can be bounded by $\frac{\e}{100}$ once $\delta$ is taken small enough.
For \eqref{eq sec7 A3}, dominating convergence theorem gives $$\int_{\R^2} \la x\ra|w_{0,\e}(X_n^{-t}(x))-w_{0,\e}(X^{-t}(x))|\,dx \le \frac{\e}{100}$$ for $n$ large enough and hence the proof of the theorem is complete.
\end{proof}
Next, we prove Theorem \ref{thm main} (iii). That is, for any $T>0$, there holds
\begin{equation*}
\sup_{t\le T} \|w_n(t)-w(t)\|_{L^1_1(\R^2)} \to 0
\end{equation*}
as $n \to \infty$.
\begin{proof}
Assume this is not true, then there exists a subsequence of $w_n$(which we still denote by $w_n$), $t_n\in [0,T]$ and $\delta>0$ such that
\begin{equation*}
\|w_n(t_n)-w(t_n)\|_{L^1_1(\R^2)} >\delta.
\end{equation*}
We may assume without loss of generality that $t_n \to t_0 \in [0,T]$ as $n \to \infty$ since $[0,T]$ is compact. Then we estimate
\begin{align}
\|w_n(t_n)-w(t_n)\|_{L^1_1(\R^2)} =&\int_{\R^2} \la x\ra|w_n(x,t_n)-w(x,t_n)|\,dx \nonumber\\
=& \int_{\R^2} \la x\ra |w_{0,n}(X_n^{-t_n}(x))-w_0(X^{-t_n}(x))|\,dx \nonumber\\
\le& \int_{\R^2} \la x\ra |w_{0,n}(X_n^{-t_n}(x))-w_0(X_n^{-t_n}(x))|\,dx \label{eq sec7 A5}\\ &+ \int_{\R^2} \la x\ra |w_0(X_n^{-t_n}(x))-w_0(X^{-t_n}(x))|\,dx\label{eq sec7 A6}.
\end{align}
For \eqref{eq sec7 A5}, in view of \eqref{eq sec4 A0}, we obtain
\begin{equation*}\begin{aligned}
\int_{\R^2} \la x\ra |w_{0,n}(X_n^{-t_n}(x))-w_0(X_n^{-t_n}(x))|\,dx &= \int_{\R^2} \la X_n(\alpha,t_n)\ra|w_{0,n}(\alpha)-w_0(\alpha)|\,d\alpha \\
&\lesssim \int_{\R^2} \la \alpha\ra|w_{0,n}(\alpha)-w_0(\alpha)|\,d\alpha \to 0.
\end{aligned}\end{equation*}
To bound the term in \eqref{eq sec7 A6}, we choose $w_{\e} \in C_c^{\infty}(\R^2)$ such that $\|w_0-w_{\e}\|_{L^1_1(\R^2)} \le \frac{\delta}{100}$. Then by a similar argument as in \eqref{eq sec7 A2}, \eqref{eq sec7 A3} and \eqref{eq sec7 A4}, there holds
\begin{equation*}\begin{aligned}
\int_{\R^2} \la x\ra |w_0(X_n^{-t_n}(x))-w_0(X^{-t_n}(x))|\,dx \le& \int_{\R^2} \la x\ra |w_0(X_n^{-t_n}(x))-w_{\e}(X_n^{-t_n}(x))|\,dx \\&+\int_{\R^2} \la x\ra |w_{\e}(X_n^{-t_n}(x))-w_{\e}(X^{-t_n}(x))|\,dx \\&+\int_{\R^2} \la x\ra |w_{\e}(X^{-t_n}(x))-w_0(X^{-t_n}(x))|\,dx \\
\le& \frac{3\delta}{100}
\end{aligned}\end{equation*}
for $n$ large enough, which contradicts the fact that
\begin{equation*}
\|w_n(t_n)-w(t_n)\|_{L^1_1(\R^2)} > \delta.
\end{equation*}
Thus the proof of Theorem \ref{thm main} (iii) is complete.

\end{proof}

\appendix
\section{}
\begin{proposition}
 For any $w \in L^1 \bigcap L^{\infty}(\R^2)$, it holds that $\nabla \cdot Hw=0$ in the sense of distribution, where
\begin{equation*}
Hw(x):=\int H(x,y)w(y)\,dy
\end{equation*}
and the kernel $H(x,y)$ is given in \eqref{eq def of H}.
\end{proposition}
\begin{proof}
Assume $w \in C_c^{\infty}(\R^2)$, then $\nabla \cdot Hw=0$ by a direct calculation as shown in Section 2. Now for general $w \in L^1 \bigcap L^{\infty}(\R^2)$, we choose a sequence of $w_{\e} \in  C_c^{\infty}(\R^2)$ such that $w_{\e} \to w$ in $L^1$ and
$|w_{\e}(x)|\lesssim 1$. Then for any $R>0$, it follows from Corollary \ref{co mei 1111} that
\begin{equation*}\begin{aligned}
|Hw(x)| &\lesssim \la x\ra^2 \int \left(1+\frac{1}{|x-y|}\right) |w(y)|\,dy \\
&\lesssim \la x\ra^2 \left( \|w\|_{L^1}+\int_{|x-y|>R} \frac{|w(y)|}{|x-y|} \,dy+\int_{|x-y|\le R} \frac{|w(y)|}{|x-y|} \,dy \right) \\
&\lesssim \la x\ra^2 \left( \|w\|_{L^1}+\frac{\|w\|_{L^1}}{R}+R\|w\|_{L^{\infty}} \right),
\end{aligned}\end{equation*}
which implies
\begin{equation}\label{eq bound for Hw by only L1 norm}
|Hw(x)|\lesssim \la x\ra^2\left(\|w\|_{L^1}+ \|w\|_{L^1}^{1/2}\|w\|_{L^{\infty}}^{1/2}\right)
\end{equation}
once we take $R=\frac{\|w\|_{L^1}^{1/2}}{\|w\|_{L^{\infty}}^{1/2}}$. Therefore, $Hw_{\e}(x) \to Hw(x)$ uniformly in every compact set since $H$ is a linear operator. Hence for any $\phi \in C_c^{\infty}(\R^2)$, we have
\begin{equation*}\begin{aligned}
\int Hw(x) \cdot \nabla \phi(x)\,dx = \lim_{\e \to 0} \int Hw_{\e}(x) \cdot \nabla \phi(x)\,dx=0.
\end{aligned}\end{equation*}
\end{proof}
Next, we will show that the particle trajectory map of $Hw$ preserves the Lebesgue measure. Before the proof, we need the following technical lemma.
\begin{lemma}\label{le app1}
Let $\{U_n(x,t)\}_{n=1}^{\infty}$ be a sequence of vector fields in $\R^d$ satisfying
\begin{enumerate}
\item $\left|U_n(x,t) \cdot \frac{x}{|x|} \right| \lesssim \la x \ra$.
\item $\sup_{x \in B_R} \left| U_n(x,t) \right| \lesssim_R 1 $.
\item $\sup_{|x|,|z|\le R}|U_n(x,t)-U_n(z,t)|\lesssim_R F(|x-z|)$.
\end{enumerate}
Let $X_{n}$ be the particle trajectory map of $U_n$, then for any $R,T >0$, $X_n$ is uniformly bounded and equicontinuous in $B_R \times [0,T]$. Similarly, let $X_n^{-t}(\cdot)$ be the inverse map of $X_n(\cdot,t)$, then $X_n^{-t}(x)$ is uniformly bounded and equicontinuous in $B_R \times [0,T]$ for any $R,T >0$.
\end{lemma}
\begin{proof} For any $|\alpha| \le R$ and $t \le T$, it is easy to check that
\begin{equation*}\begin{aligned}
|X_n(\alpha,t)|&=|\alpha|+\int_0^t \frac{X_n(\alpha,s)}{|X_n(\alpha,s)|}\cdot U_n(X_n(\alpha,s),s)\,ds \\
&\lesssim |\alpha| +\int_0^t |X_n(\alpha,s)|+1\,ds \\
&\lesssim R+T+\int_0^t |X_n(\alpha,s)|\,ds.
\end{aligned}\end{equation*}
So Gronwall's inequality implies that $X_{n}$ is uniformly bounded in $B_R \times [0,T]$.\\ In order to show the equicontinuity of $X_n$, we use assumption (iii) to estimate
\begin{equation*}\begin{aligned}
|X_n(\alpha,t)-X_n(\beta,t)|&=\int_0^t \frac{X_n(\alpha,s)-X_n(\beta,s)}{|X_n(\alpha,s)-X_n(\beta,s)|} \left( U_n(X_n(\alpha,s),s)-U_n(X_n(\beta,s),s) \right)\,ds \\
&\lesssim \int_0^t F\left( |X(\alpha,s)-X(\beta,s)| \right)\,ds.
\end{aligned}\end{equation*}
Set $z(t)=|X_n(\alpha,t)-X_n(\beta,t)|$ and assume that $z(t) \le \frac{1}{100}$ for all $t\in [0,T]$, then it follows that
\begin{equation*}
z(t) \le C\int_0^t -z(s)(\log z(s)) \,ds.
\end{equation*}
Thus comparison theorem gives
\begin{equation*}\begin{aligned}
z(t) &\le e^{e^{-Ct}\log(z(0))}\\
&\le e^{e^{-CT}\log(z(0))} \\
&=e^{e^{-CT}\log|\alpha-\beta|}.
\end{aligned}\end{equation*}
So if we take $|\alpha-\beta| \le e^{-10e^{-CT}}$, then $z(t)\le e^{-10} \le \frac{1}{100}$ for all $t \in [0,T]$ and hence
\begin{equation*}
|X_n(\alpha,t)-X_n(\beta,t)| \le e^{e^{-CT}\log|\alpha-\beta|}.
\end{equation*}
Next, we estimate
\begin{equation*}
|X_n(\alpha,t_1)-X_n(\alpha,t_2)|=\left|\int_{t_1}^{t_2} U_n(X_n(\alpha,s),s)\,ds\right|.
\end{equation*}
Recall that $X_n$ is uniformly bounded, so there exists $M>0$ such that $|X_n(\alpha,t)| \le M $ for all $\alpha \in B_R(0)$ and $t \in [0,T]$. Together with the fact that $U_n(\cdot,t)$ is uniformly bounded in $B_M\times [0,T]$, we finally obtain

\begin{equation*}
|X_n(\alpha,t_1)-X_n(\alpha,t_2)| \lesssim \int_{t_1}^{t_2}1 \,ds \lesssim |t_2-t_1|.
\end{equation*}
Therefore, for any $\delta >0$ and $n\in N^*$, there exists $\e_0=\e_0(R,T,\delta)$ such that for all $\alpha,\beta\in B_R$, $t_1,t_2 \in[0,T]$ with \ $|\alpha-\beta|+|t_2-t_1|\le \e_0$, there holds
\begin{equation*}
|X_n(\alpha,t_1)-X_n(\beta,t_2)|\le \delta.
\end{equation*}
The estimates for the inverse map $X_n^{-t}(x)$ are similar, see \cite{MaB} for references.
\end{proof}
Now we can show that the particle trajectory map of $Hw$ preserves the Lebesgue measure. More generally, we prove the following.
\begin{lemma}\label{le app A0}
Let $U(x,t)$ be a velocity field in $\R^d\times [0,T]$ with
\begin{enumerate}
\item $\left|U(x,t) \right| \le C \la x\ra^2$ for some $M>0$.
\item $\left| U(x,t) \cdot \frac{x}{\la x\ra} \right| \le C\la x\ra$.
\item $U$ is locally Log-Lipschitz continuous.
\item $\nabla \cdot U=0$ in the sense of distribution.
\end{enumerate}
Let $X(\alpha,t)$ be the particle trajectory map associated with $U$, then the map $X(\cdot,t): \R^d \to \R^d$ is bijective and preserves Lebesgue measure.
\end{lemma}
\begin{proof}
Let $\eta$ be a smooth positive function supported in $B_1$ and $$\int_{\R^d} \eta(x)\,dx=1.$$ Define
\begin{equation*}
\eta_{\e}(x)=\frac{1}{\e^{d}}\eta(\frac{x}{\e})
\end{equation*}
and
\begin{equation*}
U_{\e}(x,t):=\eta_{\e}*U(x,t)=\int_{\R^d} \eta_{\e}(x-y)U(y,t)\,dy.
\end{equation*}
Then it is easy to check that for all $\e \le 1$,
\begin{equation*}
|U_{\e}(x,t)| \lesssim \la x\ra^2.
\end{equation*}
Moreover,
\begin{equation}\label{eq app A0}\begin{aligned}
\left| U_{\e}(x) \cdot \frac{x}{\la x\ra} \right| &= \int_{\R^d} \eta_{\e}(x-y)U(y)\cdot \frac{y}{\la y \ra} \,dy +\int_{\R^d} \eta_{\e}(x-y)U(y)\cdot(\frac{x}{\la x\ra}-\frac{y}{\la y\ra}) \,dy \\
&\lesssim \la x\ra+\la x\ra^2\int_{\R^d}|x-y|\eta_{\e}(x-y)\,dy \\
&\lesssim \la x\ra+\e\la x\ra^2
\end{aligned}\end{equation}
and
\begin{equation*}\begin{aligned}
\sup_{|x|,|z|\le R}|U_{\e}(x,t)-U_{\e}(z,t)|&\lesssim \int_{\R^d} \eta_{\e}(y)| U_{\e}(x-y,t)-U_{\e}(z-y,t)|\,dy \\
&\lesssim \int_{\R^d} \eta_{\e}(y) \sup_{|x|,|z|\le R+1}|U(x,t)-U(z,t)|\,dy \\
&\lesssim_R \int_{\R^d} \eta_{\e}(y) F(|x-z|)\,dy \\
&\lesssim_R F(|x-z|).
\end{aligned}\end{equation*}

Next we estimate the particle trajectory map $X_{\e}(\alpha,t)$ associated with $U_{\e}(x,t)$:
\begin{equation*}
\la X_{\e}(\alpha,t) \ra=\la \alpha\ra+\int_0^t \frac{X_{\e}(\alpha,s)}{\la X_{\e}(\alpha,s)\ra}\cdot U_{\e}(X_{\e}(\alpha,s),s)\,ds.
\end{equation*}
Setting $z_{\e}(t)=\la X_{\e}(\alpha,t)\ra$, then it follows from \eqref{eq app A0} that
\begin{equation*}
z_{\e}(t)\le \la \alpha\ra+C\int_0^t z_{\e}(s)+\e z_{\e}(s)^2\,ds.
\end{equation*}
Assume $z_{\e}(t) \le D_1e^{D_2t}$ for all $t\in [0,T]$ and $\la \alpha\ra \le R$, then the above inequality implies that
\begin{equation*}\begin{aligned}
z_{\e}(t) &\le \la \alpha\ra+\frac{CD_1}{D_2}(e^{D_2t}-1)+\frac{\e CD_1^2}{2D_2}(e^{2D_2t}-1) \\
&\le e^{D_2t} \left[ R+1+\frac{CD_1}{D_2}+\frac{\e CD_1^2}{2D_2}e^{D_2T} \right]. \\
\end{aligned}\end{equation*}
Thus, if we take $D_2=100(\la C\ra+\la R\ra)$ and $D_1=100R+100$. Then there exists $\e_0=\e_0(R)$ such that
\begin{equation*}
z(t) \le \frac{D_1e^{D_2t}}{2}
\end{equation*}
for all $\e \le \e_0$. Therefore, a bootstrap argument then gives
\begin{equation*}
\la X_{\e}(\alpha,t)\ra \lesssim \la\alpha\ra
\end{equation*}
for all $\e \le \e_0(R)$ and $\la\alpha\ra \le R$.
Furthermore, by a similar argument as in Lemma \ref{le app1}, we see that for all $R>0$, there exists a positive constant $C_R$ such that for all $\e \le \e_0(R)$, $X_{\e}(\alpha,t)$ and $X_{\e}^{-t}(\alpha)$ are uniformly bounded and equicontinuous in $B_R \times [0,T] $ with
$|X(\alpha,t)|,|X^{-t}(\alpha)| \le C_R$. Now we take $\e \le \e_0(C_R)$ and choose subsequence $X_n(\alpha,t)$ and $X_n^{-t}(\alpha)$ converge uniformly to $X(\alpha,t)$ and $X^{-t}(\alpha)$ in $B_{C_R} \times [0,T]$, respectively.

Note that
\begin{equation*}
X_n(\alpha,t)=\int_0^t U_n(X_n(\alpha,s),s)\,ds,
\end{equation*}
Taking $n\to \infty$, we get
\begin{equation*}
X(\alpha,t)=\int_0^t U(X(\alpha,s),s)\,ds,
\end{equation*}
which implies that $X$ is the particle trajectory map of $U$ for $\la\alpha\ra \le R$. Moreover,
\begin{equation*}
X(X^{-t}(\alpha),t)=\lim_n X_n(X_n^{-t}(\alpha),t)=\alpha
\end{equation*}
for all $\la \alpha\ra \le R$. Thus, $X^{-t}$ is the inverse map of $X(\cdot,t)$ at least for $\la\alpha\ra\le R$. Now it remains to show that $X(\cdot,t)$ and $X^{-t}(\cdot)$ preserves the Lebesgue measure. Recall that $\nabla \cdot U=0$ in the sense of distribution, so $U_{\e}$ is a smooth velocity field with $\nabla \cdot U_\e=0$ and hence $X_{\e}(\cdot,t)$ and $X_{\e}^{-t}(\cdot)$ preserve Lebesgue measure. Thus the dominating convergence theorem yields
\begin{equation*}
m(X(O,t))=m(O)=m(X^{-t}(O))
\end{equation*}
for all measurable set $O\subset B_R$. Since $R$ is arbitrary, $X(\cdot,t)$ and $X^{-t}(\cdot)$ preserves Lebesgue measure. In other words,
\begin{equation*}
\int_{\R^d} f(X(\alpha,t)) \,d\alpha =\int_{\R^d} f(x) \,dx
\end{equation*}
for all $f \in L^1(\R^d)$.
\end{proof}

\begin{lemma}\label{le uniqueness of the transport equation}\label{le weak solution is lagrangian}
All weak solutions to \eqref{eq 2euler} in $L^{\infty}([0,T],L^1_1 \bigcap L^{\infty}_1)$ are indeed Lagrangian.
\end{lemma}
\begin{proof}
Assume $w\in L^{\infty}([0,T],L^1_1 \bigcap L^{\infty}_1)$, then the argument in Section 3 implies that $|Hw(x,t)|\lesssim \,\,\la x\ra^2$ and $Hw(\cdot,t)$ is locally Log-Lipschitz continuous. Set $U=Hw$ and we consider the continuity equation
\begin{equation*}
\partial_t \psi +\nabla \cdot(\psi U)=0.
\end{equation*}
We say that $\psi$ is a weak solution to the continuity equation on [0,T) with initial data $w_0$ if for all test function $\phi \in C_c^{\infty}([0,T) \times R^2)$, there holds
\begin{equation*}
-\int w_0(x)\phi(x,0)\,dx=\int_0^T \int_{\R^2} \psi (\partial_t \phi +U\cdot \nabla \phi) \,dxdt.
\end{equation*}
 Therefore, one can check directly that the solution $w(x,t)$ to the two-dimensional helical Euler equation \eqref{eq 2euler} is also a weak solution to the continuity equation. Meanwhile, Lemma \ref{le app A0} implies that $\psi(x,t)=w_0(X^{-t}(x))$ is a weak solution to the continuity equation. Since weak solutions to the continuity equation in $L^1([0,T),L^1)$ is unique \cite{Clop}, we have $w(x,t)=w_0(X^{-t}(x))$ and hence all weak solutions to \eqref{eq 2euler} are Lagrangian.
\end{proof}
\section{}
We now sketch the proof of Theorem \ref{thm existence of smooth solution}.
\begin{proof}[Proof of the existence] Recall that solutions to \eqref{eq 2euler} are unique in $L^1_1 \bigcap L^{\infty}_1(\R^2)$, so it suffices to prove the existence of weak solutions in $[0,T]$. Now we fix $T>0$ and let $w_0 \in C_c^{\infty}(\R^2)$. Assume without loss of generality that $w_0 \neq 0$ and set $w^0(x,t):=w_0(x)$. Then we can define $w^n(x,t)$ inductively:
\begin{equation*}
\left\{\begin{aligned}\partial_t w^{n+1}+Hw^n \cdot \nabla w^{n+1}=&0 \\
w^{n+1}(x,0)=&w_0(x). \end{aligned}\right.
\end{equation*}
Multiply both side by $\la x\ra$, we obtain
\begin{equation*}\begin{aligned}
\partial_t \la x\ra w^{n+1}+Hw^n \cdot \nabla (\la x \ra w^{n+1})=&w^{n+1}Hw^n \cdot \nabla \la x\ra \\
=&w^{n+1}H_1w^n \cdot \frac{x}{\la x\ra}
\end{aligned}\end{equation*}
since $Hw=H_1w+H_2w$ and $x\cdot H_2w=0$.
Note that $\nabla \cdot Hw^n=0$, so a direct calculation shows that
\begin{equation*}\begin{aligned}
\|w^{n}(\cdot,t)\|_{L^p(\R^2)} =\|w_0\|_{L^p(\R^2)}
\end{aligned}\end{equation*}
for any $p \in [1,\infty]$. Moreover,
\begin{equation*}\begin{aligned}
&\|\la x\ra w^{n+1}(\cdot,t)\|_{L^1 \bigcap L^{\infty}(\R^2)}\\ \le& \|\la x \ra w_0\|_{L^1 \bigcap L^{\infty}(\R^2)}+\int_0^t \|w^{n+1}Hw^n\|_{L^1 \bigcap L^{\infty}(\R^2)} \,ds \\
\le& \|\la x \ra w_0\|_{L^1 \bigcap L^{\infty}(\R^2)}+\int_0^t \|w^{n+1}\|_{L^{\infty}(\R^2)}\|Hw^n\|_{L^1 \bigcap L^{\infty}(\R^2)} \,ds \\
\le& \|\la x \ra w_0\|_{L^1 \bigcap L^{\infty}(\R^2)}+\|w_0\|_{L^{\infty}(\R^2)}\int_0^t \|Hw^n\|_{L^1 \bigcap L^{\infty}(\R^2)} \,ds.
\end{aligned}\end{equation*}
Thus,
$$\|w^n(t)\|_{L^1_1\bigcap L^{\infty}_1(\R^2)} \lesssim 1$$
for all $t\in [0,T]$. Together with \eqref{eq boundness for H_1w} we get
$$
|Hw^n(x)| \lesssim \la x \ra \quad \text{and} \quad \left|Hw^n(x) \cdot \frac{x}{\la x\ra}\right| \lesssim 1,
$$
which implies that $\la X^n(\alpha,t) \ra \approx \la \alpha \ra$
and hence there exists $M>0$ such that $w^n(\cdot,t)$ supported in $B_M(0)$ for all $n>0$ and $t \in [0,T]$.
Furthermore, Lemma \ref{le sec3 A0} yields
$$
|Hw^n(x,t)-Hw^n(z,t)|\lesssim (\la x\ra +\la z\ra) F(|x-z|)
$$
for any $R>0$ and $t \in [0,T]$. So in view of Lemma \ref{le app A0}, there exist $X(\alpha,t)$, $\tilde{X}(\alpha,t)$ and a subsequence $n_k$ such that $X_{n_k}$ and $X_{n_k-1}$  converge uniformly to $X$ and $\tilde{X}$, respectively in compact subset of $\R^2 \times [0,T]$. Note that $w^{n_k+1}$ satisfies the transport equation
\begin{equation*}
\left\{\begin{aligned}\partial_t w^{n_k+1}+Hw^{n_k} \cdot \nabla w^{n_k+1}=&0 \\
w^{n_k+1}(x,0)=&w_0(x). \end{aligned}\right.
\end{equation*}
So letting $k \to \infty$, by a similar argument as in Section 5, we find that $w$ satisfies
\begin{equation*}
\left\{\begin{aligned}\partial_t w+H\tilde{w} \cdot \nabla w=&0 \\
w(x,0)=&w_0(x), \end{aligned}\right.
\end{equation*}
where $w(x,t):=w_0(X^{-t}(x))$ and $\tilde{w}(x,t):=w_0(\tilde{X}^{-t}(x))$. To complete the proof, it remains to show that $w \equiv \tilde{w}$, which is equivalent to $X(\alpha,t)=\tilde{X}(\alpha,t)$ for every $\alpha \in$ supp $w_0$. Therefore, we define the distance
\begin{equation*}
D_n(t):=\int |X^{n+1}(\alpha,t)-X^n(\alpha,t)| |w_0(\alpha)|\,d\alpha.
\end{equation*}
 Observe that $D_n(t)$ is uniformly bounded since $w_0$ has compact support. So by dominating convergence theorem, it suffices to show that for any $t\in [0,T]$, $\lim_{n\to \infty} D_n(t) = 0$. First, we use the Newton-Leibniz formula to conclude that
\begin{align}
D_n(t)\le&\int_0^t \int  \left| Hw^{n+1}(X^{n+1}(\alpha,s),s)-Hw^{n}(X^{n}(\alpha,s),s) \right||w_0(\alpha)|\,d\alpha\,ds\nonumber\\
\le&\int_0^t \int  \left| Hw^{n+1}(X^{n+1}(\alpha,s),s)-Hw^{n+1}(X^{n}(\alpha,s),s) \right||w_0(\alpha)|\,d\alpha\,ds\label{eq app B1} \\
&+\int_0^t \int  \left| Hw^{n+1}(X^{n}(\alpha,s),s)-Hw^{n}(X^{n}(\alpha,s),s) \right||w_0(\alpha)|\,d\alpha\,ds. \label{eq app B2}
\end{align}
For \eqref{eq app B1}, inequality \eqref{eq hw(x)-hw(z) 1}, \eqref{eq hw(x)-hw(z) 3} and Lemma \ref{le sec3 A0} yield
\begin{equation*}\begin{aligned}
&\int_0^t \int  \left| Hw^{n+1}(X^{n+1}(\alpha,s),s)-Hw^{n+1}(X^{n}(\alpha,s),s) \right||w_0(\alpha)|\,d\alpha\,ds \\
\lesssim& \int_0^t F(D_{n+1}(s))\,ds,
\end{aligned}\end{equation*}
where we have used the Jensen's inequality and the fact that $w_0 \neq 0$. For \eqref{eq app B2}, note that $X^n(\cdot,t)$ preserves Lebesgue measure, so it follows that
\begin{equation*}\begin{aligned}
&\int_0^t \int  \left| Hw^{n+1}(X^{n}(\alpha,s),s)-Hw^{n}(X^{n}(\alpha,s),s) \right||w_0(\alpha)|\,d\alpha\,ds \\
=& \int_0^t \int  \left| Hw^{n+1}(x,s)-Hw^{n}(x,s) \right||w^{n+1}(x,s)|\,dx\,ds\\
=& \int_0^t \int \left| \int H(x,y)w^{n+1}(y,s)\,dy-\int H(x,y)w^n(y,s)\,dy \right||w^{n+1}(x,s)|dxds.
\end{aligned}\end{equation*}
Recall that $w^n(X^{n-1}(\alpha,t),t)=w_0(\alpha)$, so we have
\begin{equation*}\begin{aligned}
&\int_0^t \int  \left| Hw^{n+1}(X^{n}(\alpha,s),s)-Hw^{n}(X^{n}(\alpha,s),s) \right||w_0(\alpha)|\,d\alpha\,ds \\
=& \int_0^t \int \left| \int \left(H(x,X^n(\beta,s))-H(x,X^{n-1}(\beta,s))\right)|w_0(\beta)|\,d\beta \right||w^{n+1}(x,s)|dxds.
\end{aligned}\end{equation*}

Thus, \eqref{eq sec2 B1}, \eqref{eq def of H}, \eqref{eq def of H1}, \eqref{eq def of H2}, \eqref{eq hw(x)-hw(z) 1}, \eqref{eq sec3 B6} and Lemma \ref{le sec3 A0} yield
\begin{equation*}\begin{aligned}
&\int_0^t \int  \left| Hw^{n+1}(X^{n}(\alpha,s),s)-Hw^{n}(X^{n}(\alpha,s),s) \right||w_0(\alpha)|\,d\alpha\,ds \\
\lesssim& \int_0^t F(D_{n}(s))\,ds
\end{aligned}\end{equation*}
since $w(\cdot,t)$ supported in $B_M(0)$ for all $t \in [0,T]$. Gathering the estimates above, we arrive at
\begin{equation*}
D_n(t)\lesssim \int_0^t F(D_n(s))\,ds +\int_0^t F(D_{n-1}(s))\,ds.
\end{equation*}
Let $\tilde{D}_n(t):=\sup_{m\ge n} D_m(t)$, then it follows that
$$\tilde{D}_n(t) \lesssim \int_0^t F(\tilde{D}_{n-1}(s))\,ds.$$
Setting $D^*(t):=\lim_{n\to \infty}\tilde{D}_n(t)$ and letting $n \to +\infty$, we have
$$
D^*(t) \lesssim \int_0^t F(D^*(s))\,ds,
$$
which implies $D^* \equiv 0$ by Osgood's Lemma. Observe that $D_n(t) \le \tilde{D}_n(t)$, so we finally get
$$
\lim_{n\to \infty} D_n(t)=0
$$
and hence completes the proof.
\end{proof}

\end{document}